\documentclass[12pt]{amsart}

\usepackage{xy, graphicx, color, hyperref}
 
 \xyoption{all}

\usepackage{amsmath}
\usepackage{amssymb}
\usepackage{amsfonts, ytableau  ,hyperref}
 
\usepackage{tikz-cd} 

\usepackage{tikz,verbatim}

\makeatletter
\def\subsection{\@startsection{subsection}{3}%
  \z@{.9\linespacing\@plus.7\linespacing}{.1\linespacing}%
  {\normalfont\bfseries}}
\makeatother 
 
\title[Determinants of Representations of Coxeter Groups]{Determinants of Representations of Coxeter Groups}
\author{Debarun Ghosh}
\author{Steven Spallone}

\newtheorem{theorem}{Theorem}
\newtheorem{corollary}{Corollary}
\newtheorem{lemma}{Lemma}

\newtheorem{proposition}{Proposition}
\newtheorem{defn}{Definition}

\newcommand{\C}{\mathbb C}

\newcommand{\nc}{\newcommand}

\nc{\prop}{\proposition}
\nc{\thm}{\theorem}
\nc{\cor}{\corollary}
\nc{\mc}{\mathcal}
\nc{\mb}{\mathbb}
\nc{\mf}{\mathfrak}
\nc{\ul}{\underline}
\nc{\ol}{\overline}
\nc{\N}{\mb N}
\nc{\R}{\mb R}
\nc{\Z}{\mb Z}
\nc{\Q}{\mb Q}

\nc{\dmo}{\DeclareMathOperator}

\nc{\domeq}{\overset{D}{=}}

\dmo{\Ker}{Ker} \dmo{\val}{val} \dmo{\ord}{ord}
\dmo{\Int}{Int}
 
 \dmo{\neat}{neat}
 \dmo{\messy}{messy}
 
\dmo{\I}{I}
\dmo{\II}{II}
\dmo{\odd}{odd}
\dmo{\sgn}{sgn}
\nc{\beq}{\begin{equation*}}
\nc{\eeq}{\end{equation*}}
\nc{\half}{\frac{1}{2}}
\dmo{\Mod}{mod}
\dmo{\ab}{ab}
\dmo{\ver}{ver}
\dmo{\core}{core}
\dmo{\quo}{quo}
\dmo{\bin}{bin}
\dmo{\Bip}{Bip}

\dmo{\res}{res}
\dmo{\lin}{lin}
 
\dmo{\Sp}{Sp}
\dmo{\SO}{SO}
\dmo{\Irr}{Irr}
 \dmo{\Spin}{Spin}
\dmo{\GSp}{GSp}
\nc{\la}{\lambda}
  
 \nc{\lip}{\langle}
 \nc{\rip}{\rangle}
\nc{\eps}{\varepsilon}

\dmo{\Perm}{Perm}
\dmo{\Res}{Res}
\dmo{\Ind}{Ind}
\dmo{\tr}{tr}
\dmo{\Sym}{Sym}
\dmo{\reg}{reg}
\dmo{\summ}{sum}
\dmo{\GL}{GL}

\begin{document}
\maketitle

\begin{abstract}
In \cite{APS-chiral}, the authors characterize the partitions of $n$ whose corresponding representations of $S_n$ have nontrivial determinant.
The present paper extends this work to all irreducible finite Coxeter groups $W$.  Namely, given a nontrivial multiplicative character $\omega$ of $W$, 
we give a closed formula for the number of irreducible representations of $W$ with determinant $\omega$.  For Coxeter groups of type $B_n$ and $D_n$, this is accomplished by characterizing the bipartitions associated to such representations.
 \end{abstract}

\tableofcontents 

\section{Introduction}
 
 The tools of this paper, and \cite{APS-chiral}, have their genesis in the paper \cite{macdonald} of Macdonald, who developed the arithmetic of partitions to give a closed formula for the number $A(n)$ of odd-dimensional Specht modules (irreducible representations) of $S_n$.  Briefly, if $n$ is expressed in binary as a sum of powers of $2$, then $A(n)$ is the product of those powers of $2$.  This formula comes from characterizing the $2$-core tower of ``odd" partitions; we review this notion in Section 5.1.  In \cite{APS-chiral}, the authors describe a simple way to read off the determinant of a representation of $S_n$ from its $2$-core tower, and in particular give a closed formula for the number $B(n)$ of Specht modules whose determinant is the sign character.
Since the number $p(n)$ of partitions of $n$ itself does not have a nice closed formula, it is remarkable that such formulas for $A(n)$ and $B(n)$ exist. 
 
Most of this paper is devoted to the Coxeter group of type $B_n$, which we write as `$\mb B_n$' for simplicity.  These are the ``hyperoctahedral groups".  Here $\mb B_n$ is the wreath product of $\Z/2 \Z$ by $S_n$.  There are four multiplicative characters of $\mb B_n$, which we denote by $\omega=1,\sgn^0, \eps$, and $\sgn^1$, described as follows.  The character $\sgn^0$ is the composition of the projection to $S_n$ with $\sgn$, the character $\eps$ is $(-1)^m$, where $m$ is the sum of the entries of the $(\Z/2\Z)^n$-factor, and finally $\sgn^1=\eps \cdot \sgn^0$.

Irreducible representations of $\mb B_n$, written  $\rho_{\alpha \beta}$, are parametrized by pairs $(\alpha,\beta)$ of partitions with $|\alpha|+|\beta|=n$, called bipartitions of $n$.   The representations are induced by certain ``Young Subgroups" $\mb B_a \times \mb B_b \leq \mb B_n$.  For each $\omega$, we characterize the $2$-core towers of $\alpha,\beta$ so that $ \det \rho_{\alpha \beta}=\omega$.  For $\omega \neq 1$, this leads to closed formulas for
\beq
N_\omega(n)= \# \{ (\alpha,\beta) : |\alpha|+|\beta|=n, \det \rho_{\alpha \beta}=\omega \}.
\eeq 

 Moreover we prove:
  
 \begin{theorem}
\label{inequality}
Let $n \geq 10$.  
\begin{enumerate}
\item If $n$ is odd, then $N_{\eps}(n)=N_{\sgn^1}(n)<N_{\sgn^0}(n)<N_{1}(n)$. 
\item
 If $n$ is even, then $N_{\eps}(n)=N_{\sgn^0}(n)<N_{\sgn^1}(n)<N_{1}(n)$.
\end{enumerate}
\end{theorem}

In Section 2, after establishing notations, we give formulas for the determinant of a representation of a Coxeter group in terms of the character.  In Section 3, for completeness, we indicate the solution to the determinant problem for dihedral groups.  Section 4 reviews the work of \cite{macdonald} and   \cite{APS-chiral} for $S_n$, and develops it further for application to the hyperoctahedral case.  

Most of the work is in Section 5, where we first compute $\det \rho_{\alpha \beta}$ in terms of known quantities (`$f_\alpha$' the degree of the Specht module $\rho_\alpha$, and `$g_\alpha$' the multiplicity of $-1$ as an eigenvalue of $\rho_\alpha((12))$, and similarly for $\beta$).  We use a well-known formula for the determinant of an induced representation.  We then provide a table and a logplot illustrating the values of $N_\omega(n)$ for small $n$.
This is followed by a conceptual proof of the identity $N_{\sgn^1}(n)=N_\eps(n)$ when $n \geq 3$ is odd.  No such ``pure thought" proof is apparent for the other identity $N_{\sgn^0}(n)=N_\eps(n)$ when $n$ is even.  Next, we calculate each
\beq
N_\omega(a,b)=\{ (\alpha,\beta): |\alpha|=a,|\alpha|=b, \det \rho_{\alpha \beta}=\omega\}.
\eeq
In Section 5.6, we compute $N_\omega(n)=\sum_{a+b=n} N_\omega(a,b)$.  Considerable work goes into simplifying these sums to arrive at closed formulas.   These final formulas take the form
\beq
N_\omega(n)=A(2n) \times f_\omega(k),
\eeq
where $f_\omega(k)$ is a function only depending on $k=\ord_2(n)$, for $n$ even, and similarly in terms of $\ord_2(n-1)$ when $n$ is odd.

Next we apply Clifford Theory to the Coxeter group of type  $D_n$.  This group is the kernel of $\eps$ and we denote it by $\mb D_n$ (not to be confused with the dihedral groups, of type $I_2(n)$).  This group has two multiplicative characters, the nontrivial one being the restriction of $\sgn^0$, which we denote as `$\sgn$'.   We give a closed formula for $N'_{\sgn}(n)$, the number of irreducible representations of $\mb D_n$ with determinant $\sgn$, in terms of our earlier formulas for $N_\omega(n)$.

We similarly treat the remaining (exceptional) types of finite irreducible Coxeter groups.  Given the formulas of Section 2, this comes to a finite calculation with the character tables of \cite{Geck}.

{\bf Acknowledgements:} We would like to thank Dipendra Prasad and Amritanshu Prasad for their interest and useful conversations.
 This research was driven by computer exploration using the open-source
mathematical software \texttt{Sage}~\cite{sage} and its algebraic
combinatorics features developed by the \texttt{Sage-Combinat}
community~\cite{Sage-Combinat}. 
 
 \section{Notation and Preliminaries}

 \subsection{Binary Notation} \label{bin.notation}
 
 The nonnegative integer $n$ and its binary digits play a substantial role in this paper, so it is convenient to fix notation here.  Put
 \begin{equation} \label{n}
 n= \epsilon + 2^{k_1} + \cdots + 2^{k_r},
 \end{equation}
 with $\epsilon$ either $0$ or $1$, and $1 \leq k_1<k_2 < \cdots < k_r$.  We will sometimes write $k=k_1$, and $n'=n-\epsilon$.
 Let $\nu(n)$ denote the number of $1$'s in the binary expansion of $n$; in the above $\nu(n)=r+\epsilon$.
 
 Write $\bin(n)$ for the set of powers of $2$ in the binary expansion of $n$; in the above $\bin(n)=\{k_1, \ldots, k_r \}$, together with $0$ if $n$ is odd.
 
 Given a natural number $a$, write $\ord(a)$ for the highest power of $2$ dividing $a$.  In the above, $k=k_1=\ord(n')$.  (It is traditional to write `$\ord_2(n)$', but in this paper only `$2$' matters.)  Note for later that if $0 \leq j \leq k$, then
 \begin{equation} \label{nu.formula}
 \nu(n-2^j)=\nu(n)+k-j-1.
 \end{equation}

  \begin{defn}
Let $a,b$ be nonnegative integers. We say ``\thinspace  $\summ(a,b)$ is neat", provided that there is no carry when adding $a$ to $b$ in binary.  We also write `$a+b \doteq n$' to denote that $a+b=n$ and the sum is neat.  Similarly we say ``\thinspace $\summ(a,b)$ is messy", provided that there is a carry.
 \end{defn}

 Thus $\summ(a,b)$ is neat iff $\bin(a)$ and $\bin(b)$ are disjoint.  It is a classical observation that the binomial coefficient $\binom{n}{a,b}=\frac{n!}{a! b!}$ is odd iff $\summ(a,b)$ is neat; we will make extensive use of this   without further comment.  More generally, $\ord \binom{n}{a,b}$
 is equal to the number of carries when adding $a$ to $b$ in binary.

 \subsection{Partition Notation}

The notation `$\la \vdash n$' means that $\la$ is a partition of $n$.  We also write $|\la|=n$ in this case.  Write  $\la'$ for the conjugate partition; its Young diagram is the transpose of the Young diagram of $\la$.  A ``hook" is a partition of the form $\la=(a,1^b)$ for some nonnegative integers $a,b$.
 
A bipartition of a number $n$ is a pair of partitions $(\alpha,\beta)$ with $|\alpha|+|\beta|=n$. 
 The notation `$(\alpha,\beta) \models n$' means that $(\alpha,\beta)$ is a bipartition of $n$. 
 
 Write $p(n)$ for the number of partitions of $n$, and $p_2(n)$ for the number of bipartitions of $n$.
 Write $\Lambda$ for the set of partitions, $\Lambda_n$ for the set of partitions of $n$, and $\Bip(n)$ for the set of bipartitions of $n$.
 
 
 \subsection{Group Theory Notation}
 
 For a group $G$, write $D(G)$ for its derived (commutator) subgroup, and $G_{\ab}$ for its abelianization $G/D(G)$.

 There are four infinite families of finite irreducible Coxeter groups, namely ones of type $A_n$, $B_n$, $D_n$, and $I_2(n)$ (the dihedral groups).  
 The exceptional types are $E_6$, $E_7$, $E_8$, $F_4$, $H_3$, $H_4$.  (See \cite{Bou.Lie.4-6}.)
 
 All representations considered in this paper are finite-dimensional and complex.  
 Write $\Irr(G)$ for the set of isomorphism classes of irreducible representations of $G$.   If $\pi$ is a representation, we write $\chi_\pi$ for its character.

 By ``multiplicative character" we mean a group homomorphism $W \to \C^{\times}$, which since $W$ is generated by involutions, must take values in $ \{ \pm 1 \}$.  For a representation $(\pi,V)$   of a group $G$, write $\det \pi$ for the composition of $\pi: G \to \GL(V)$ with the determinant map.  Then $\det \pi$ is a multiplicative character of $G$.

 If $(\pi_1,V_1)$ and $(\pi_2,V_2)$ are representations of groups $G_1$ and $G_2$, write $\pi_1 \boxtimes \pi_2$ for the external tensor product representation of $G_1 \times G_2$ on $V_1 \otimes_{\C} V_2$.
 
 \subsection{Solomon Principle} \label{SolomonP}
 
 Let $\pi$ be a representation of a Coxeter group $W$.  In this section we show how to infer $\det \pi$ from its character.  So let $(W,S)$ be a Coxeter group ($S$ is a certain set of generators of order $2$; see \cite{Bou.Lie.4-6}).   There is a unique multiplicative character $\eps_W$ so that $\eps_W(s)=-1$ for each $s \in S$, namely
\beq
\eps_W(w)=(-1)^{\ell(w)},
\eeq
where $\ell(w)$ is the length of $w$ with respect to $S$.

For the Coxeter groups of type $A_n$, $D_n$, $E_6$, $E_7$, $E_8$, $H_3$, $H_4$, and $I_2(p)$ with $p$ odd, the trivial character and $\eps_W$ are the only multiplicative characters.  This is equivalent to the abelianization $W_{\ab}$ having order  $2$.

\begin{proposition} \label{solomon.principle} Suppose $|W_{\ab}|=2$ , and let $s \in S$.  If $\pi$ is a representation of $W$, then 
\beq
\det \pi=(\eps_W)^b,
\eeq
where
\beq
b=\frac{\dim \pi-\chi_\pi(s)}{2}.
\eeq
\end{proposition}

\begin{proof} Let $a$ be the multiplicity of $1$ as an eigenvalue of $\pi(s)$, and $b$ be the multiplicity of $-1$.  Then $\dim \pi=a+b$ and $\chi_\pi(s)=a-b$.
\end{proof}
 
 We attribute this approach to L. Solomon; see \cite{ec2}, Exercise 7.55.
 
The abelianization of the other Coxeter groups ($B_n$, $F_4$, and $I_2(p)$ for $p$ even) are Klein $4$ groups.  For these, fix two non-conjugate simple reflections $s_1,s_2 \in S$, and multiplicative characters $\omega_1,\omega_2$ so that $\omega_1(s_1)=-1$, $\omega_1(s_2)=1$, $\omega_2(s_1)=1$, $\omega_2(s_2)=-1$.  Then $\eps_W=\omega_1 \cdot\omega_2$ and the multiplicative characters of $W$ are $\{ 1, \omega_1,\omega_2, \omega_1 \cdot\omega_2 \}$.  
 
 \begin{proposition} \label{solomon.principle2} Suppose $|W_{\ab}|=4$, and let $s_1,s_2,\omega_1,\omega_2$ be as above.  If $\pi$ is a representation of $W$, then 
\beq
\det \pi=(\omega_1)^{x_1}(\omega_2)^{x_2},
\eeq
where
\beq
x_1=\frac{\dim \pi-\chi_\pi(s_1)}{2},
\eeq
and
\beq
x_2=\frac{\dim \pi-\chi_\pi(s_2)}{2}.
\eeq
\end{proposition}

\begin{proof} Similar to the proof of Proposition \ref{solomon.principle}.
\end{proof}

 \section{Type $I_2(p)$: Dihedral Groups}
 
 Let us sketch the case of dihedral groups.  Let $p \geq 1$ be a positive integer, and let $W=D_p$ be the dihedral group of order $2p$.  
 The irreducible representations of $W$ are well-known: they are either one-dimensional, or  induced from the normal cyclic subgroup of order $p$, thus two-dimensional.
 It is elementary to check that the determinant of each two-dimensional irreducible representation is $\eps_W$.
 
 \bigskip
 
 {\bf Case $p$ odd:}  In this case $|W_{\ab}|=2$, so $1$ and $\eps_W$ are the only multiplicative characters.    From the above,
 \beq
N_{1}(p)=\#\{\pi \in\Irr(D_p)\mid\det \pi=1\}=1
\eeq
and 
\beq{}
N_{\eps_W}(p)=\#\{\pi \in\Irr(D_p)\mid\det \pi=\eps_W\}=\frac{p+1}{2}.
\eeq{}
 
{\bf Case $p$ even:} Now  $|W_{\ab}|=4$, so there are four multiplicative characters.   We have $N_{\eps_W}(p)=\dfrac{p}{2}$ and $N_\omega(p)=1$ for $\omega \neq \eps_W$.

 \section{Type $A_n$: Symmetric Groups}
 
 In this section we review and develop material from  \cite{macdonald}, \cite{olsson}, and \cite{APS-chiral} for $S_n$, which we will apply to the Type $B_n$ Coxeter groups in the next section.  The irreducible representations of $S_n$ are indexed by partitions $\lambda$ of $n$.  Given $\la \vdash n$, we denote by $(\rho_\la,V_\la)$ the corresponding representation of $S_n$.   
 
 \subsection{Cores, Quotients, and Towers} \label{tower.sec}

 Let us quickly recall the theory of $2$-cores and $2$-quotients of partitions.  Details may be found in \cite{olsson}.
   \begin{defn}  We say that a partition $\la$ is a $2$-core partition, provided that none of its hook lengths are even.  The set of $2$-core partitions is denoted by $\mc C_2$.
  \end{defn}
    
    There is assigned to any partition $\la$ a partition $\core_2 \la \in \mc C_2$ called the  $2$-core of $\la$,  and a bipartition $(\alpha,\beta)$, called the $2$-quotient of $\la$.   These have the property that 
    \begin{equation} \label{song}
    |\la|=|\core_2 \la|+2 (|\alpha|+|\beta|).
    \end{equation}
       Moreover, this assignment is a bijection between $\Lambda$ and $\mc C_2 \times \Lambda \times \Lambda$.     
 In what follows, we iterate the core-quotient procedure to obtain what is called the $2$-core tower of $\lambda$.
 This is an infinite binary tree, with each node labeled with a $2$-core partition.  For uniformity, let us put $\lambda_0=\alpha$ and $\lambda_1=\beta$.
 The tower is organized into rows, one for each nonnegative integer.
 Its $0$th row is the root of the tree, labeled with $\alpha_\emptyset := \core_2\lambda$.  Its first row comprises two nodes
 \beq
  \alpha_0, \alpha_1,
\eeq
where, if $\quo_2\lambda = (\lambda_0,\lambda_1)$, then $\alpha_i = \core_2{\lambda_i}$.
Let $\quo_2{\lambda_i} = (\lambda_{i0},\lambda_{i1})$, and define $\alpha_{ij} = \core_2{\lambda_{ij}}$.
The second row  is 
\begin{displaymath}
  \alpha_{00}, \alpha_{01}, \alpha_{10}, \alpha_{11}.
\end{displaymath}

Recursively, having defined partitions $\lambda_x$ for a binary sequence $x$, define the partitions $\lambda_{x0}$ and $\lambda_{x1}$ by
\begin{equation}
  \label{eq:6}
  \quo_2{\lambda_x} = (\lambda_{x0}, \lambda_{x1}),
\end{equation}
and let $\alpha_{x\epsilon} = \core_2{\lambda_{x\epsilon}}$ for $\epsilon=0,1$.
The $i$th row of the $2$-core tower of $\lambda$ consists of the partitions $\alpha_x$, where $x$ runs over the set of all $2^i$ binary sequences of length $i$, listed from left to right in lexicographic order.

Thus the $2$-core tower is:
\begin{displaymath}
  \resizebox{.9\hsize}{!}{
  \xymatrix@C=0em{
    &&&&&&& \alpha_{\emptyset}\ar@{-}[dllll]\ar@{-}[drrrr] &&&&&&&\\
    &&&\alpha_{0}\ar@{-}[dll]\ar@{-}[drr] &&&&&&&& \alpha_{1} \ar@{-}[dll]\ar@{-}[drr] &&&\\
    &\alpha_{00}\ar@{-}[dl]\ar@{-}[dr]&&&& \alpha_{01}\ar@{-}[dl]\ar@{-}[dr] &&&&\alpha_{10}\ar@{-}[dl]\ar@{-}[dr]&&&&\alpha_{11}\ar@{-}[dl]\ar@{-}[dr]&\\
    \alpha_{000}&&\alpha_{001}&&\alpha_{010}&&\alpha_{011}&&\alpha_{100}&&\alpha_{101}&&\alpha_{110}&&\alpha_{111} \\
   &&&&&&&  \vdots
  }}
\end{displaymath}

This is called the $2$-core tower of $\la$, which we denote by `$T_\la$'.  Note that $T_\la$ has nonempty partitions at only finitely many nodes.  Moreover $\la$ can be reconstructed from $T_\la$.

If $w_i(\la)$ is the sum of the sizes of the partitions in the $i$th row of $T_\la$, then iterating (\ref{song}) gives
\begin{equation} \label{weights.tower}
|\la| = \sum_{i=0}^\infty w_i(\la)2^i.
\end{equation}

{\bf Example:}  Let $\la=(12,2,1,1)$.  The $2$-core of $\la$ is empty and its $2$-quotient is $((1,1), (6))$.
The $2$-core of $(1,1)$ is also empty, and its $2$-quotient is $((1),\emptyset)$.  The partition $(6)$ has empty $2$-core, and $2$-quotient  $(\emptyset,(3))$.  
Finally, $(3)$ has $2$-core $(1)$ and $2$-quotient $((1),\emptyset)$.  From this we can write down $T_\la$:

\begin{displaymath}
  \xymatrix@C=0.7em{
    &&&&&&& \emptyset \ar@{-}[dllll]\ar@{-}[drrrr] &&&&&&&\\
    &&&\emptyset\ar@{-}[dll]\ar@{-}[drr] &&&&&&&& \emptyset \ar@{-}[dll]\ar@{-}[drr] &&&\\
    &(1) \ar@{-}[dl]\ar@{-}[dr]&&&& \emptyset \ar@{-}[dl]\ar@{-}[dr] &&&& \emptyset \ar@{-}[dl]\ar@{-}[dr]&&&& (1) \ar@{-}[dl]\ar@{-}[dr]&\\
    \emptyset &&\emptyset&&\emptyset&&\emptyset&&\emptyset&&\emptyset&&(1) &&\emptyset \\
  &&&&&&&  \vdots
  }
\end{displaymath}

  The fourth row and below are empty.
  
  We also define a map $\phi: \Lambda_a \times \Lambda_b \to \Lambda_{2(a+b)}$ so that $\phi(\alpha,\beta)$ has trivial $2$-core, and its $2$-quotient is $(\alpha,\beta)$. Note $T_{\phi(\alpha,\beta)}$ is obtained simply by joining $T_\alpha$ and $T_\beta$ side-by-side with root node labeled `$\emptyset$'.  For example, $(12,2,1,1)=\phi((1,1), (6))$ and $(6)=\phi(\emptyset,(3))$.   All partitions of an even number $n$ with trivial $2$-core are uniquely of the form $\phi(\alpha,\beta)$ for some $(\alpha,\beta) \models \frac{n}{2}$.

 \subsection{Review of Macdonald}
 
 \begin{defn}  Write $f_\la$ for the dimension of $V_\la$.  Say that a partition $\la$ is odd, provided that $f_\la$ is odd.  Otherwise say that $\la$ is even.  Write $\Lambda^{\odd}$ for the set of odd partitions, and
 $\Lambda^{\odd}_n$ for the set of odd partitions of $n$.
Write $A(n)=|\Lambda^{\odd}_n|$.
 \end{defn}
 
 The main result of \cite{macdonald} is:
 \begin{theorem} A partition is odd precisely when $w_i(\la) \leq 1$ for all $i \geq 0$.  \end{theorem}

In other words, it is odd when its $2$-core tower has at most one cell in each row.  For example, from the towers we see that the partitions $(12,2,1,1)$ is even, while the partitions $(1,1)$ and $(6)$ are odd.

By (\ref{weights.tower}), if $\la \vdash n$ is an odd partition, then $\bin(n)$ is equal to the set of indices of the nontrivial rows of $T_\la$.

\begin{cor}   For $n$ with binary expansion 
 \begin{equation*}
 n= \epsilon + 2^{k_1} + \cdots + 2^{k_r},
 \end{equation*}
 put
\beq{}
\alpha(n) = k_1+k_2 + \cdots + k_r.
\eeq{}
Then $A(n)=2^{\alpha(n)}$.  
\end{cor}

In other words, if we express $n$ in binary as a sum of powers of $2$, then $A(n)$ is the {\it product} of those powers of $2$.
In particular, if $n$ is a power of $2$, then $A(n)=n$.  In this case the odd partitions of $n$ are precisely the hooks of length $n$.

Let us record some properties of $A(n)$ for later use.

\begin{lemma} \label{alpha.increment}
\begin{enumerate}
\item If $x+y \doteq n$, then $A(n)=A(x)A(y)$.
\item If $n$ is odd, then $A(n-1)=A(n)$.
\item More generally, if $0 \leq j \leq k=\ord(n')$, then
\beq
A(n-2^j)= 2^{-k+ \binom{k}{2}-\binom{j}{2}}A(n).
\eeq
\item For all $n$ we have
\beq
A(2n)=2^{\nu(n)}A(n).
\eeq
\end{enumerate}
\end{lemma}

\begin{proof}
For the third part, we have
\beq
\begin{split}
\alpha(n)& =\alpha(n-2^k)+\alpha(2^k) \\
		&=\alpha(n-2^k)+k,\\
		\end{split}
		\eeq
and
\beq
\alpha(n-2^j)=\alpha(n-2^k)+\alpha(2^k-2^j).
\eeq
Now
\beq
\begin{split}
\alpha(2^k-2^j) &=j+(j+1)+ \cdots + (k-1) \\
			&=\binom{k}{2}-\binom{j}{2}. \\
			\end{split}
			\eeq
			Thus
			\beq
\alpha(n-2^j)=\alpha(n) -k+ \binom{k}{2}-\binom{j}{2}.
\eeq
			
The result follows.  The other parts are straightforward.  Also see Proposition \ref{merge1} for a bijective proof of the first part.
\end{proof}

  Here are simple bounds for $A(n)$:  
  \begin{lemma} \label{A(n).bound} Let $n$ be a natural number.  We have
\beq
\half n < A(n) \leq n^{\half (\log_2n+1)}.
\eeq
\end{lemma}
\begin{proof}
We have
\beq
\begin{split}
\log_2 n-1 < [\log_2 n] \leq \alpha(n) & \leq 1+2 + \cdots + [\log_2 n] \\
&= \frac{[\log_2n]([\log_2n]+1)}{2} \\
							&\leq \half \log_2n (\log_2n +1). \\
\end{split}
\eeq
Since $A(n)=2^{\alpha(n)}$, this gives the lemma.
\end{proof}

 \subsection{Review of Ayyer-Prasad-Spallone}

 Put $s_1=(12) \in S_n$ for $n \geq 2$.
 
 \begin{defn} 
 Given $\la \vdash n$, put 
 \beq
 g_\la=\frac{f_\la-\chi_\la(s_1)}{2}.
 \eeq
 Say that a partition $\la$ is chiral, provided that $\det \rho_\la =\sgn$.
  Write $B(n)$ for the number of chiral partitions of $n$.
 \end{defn}
 By Proposition \ref{solomon.principle}, $\la$ is chiral iff $g_\la$ is odd.
 Recall the notation `$\ord$', `$\bin$', and $k=\ord(n')$ from Section \ref{bin.notation}.  

\begin{theorem} \label{APS-chiral}
\cite[Theorem 6]{APS-chiral} 
  Let $\la$ be a partition with $2$-quotient $(\alpha,\beta)$ with $|\alpha|=a$ and $|\beta|=b$.
Then $\lambda$ is chiral if and only if one of the following holds:
\label{chiral}
\begin{enumerate}
\item   $\la$ is odd, and 
\begin{enumerate}
\item if $n$ is even, then $k-1 \in \bin(a)$.
\item if $n$ is odd, then $k-1 \in \bin(b)$.
\end{enumerate}
 \item $\core_2 \la =\emptyset$ or $(1)$, and
\begin{enumerate}
\item $\alpha$ and $\beta$ are odd, 
\item $\bin(a) \cap \bin(b)= \{ j \}$, with $j=\ord(a)=\ord(b)$.   
\end{enumerate}
\item $\core_2 \la=(2,1)$ and $\phi(\alpha, \beta)$ is odd.    
\end{enumerate}
\end{theorem}

 \begin{corollary} \label{f.g.correlation}
\begin{enumerate}
\item $\# \{ \la \vdash n \mid f_{\lambda} \equiv g_{\lambda} \equiv 1 \}=\frac{1}{2}A(n)$. 
\item $\# \{ \la \vdash n \mid f_{\lambda} \equiv 1, g_{\lambda} \equiv 0 \}=\frac{1}{2}A(n)$. 
\item $\# \{ \la \vdash n \mid f_{\lambda} \equiv 0, g_{\lambda} \equiv 1 \}=B(n)-\frac{1}{2}A(n)$. 
\end{enumerate}
\end{corollary}

Let us give special attention to the chiral partitions of type (2) in Theorem \ref{APS-chiral}.  The corresponding towers have at most one cell in each row, except for the $j$th row which has precisely two cells,   one   in the left half, and another in the right half.  Necessarily $j \geq 1$. Moreover, there are no cells in rows $1$ through $j-1$.  Its root is labeled $\emptyset$ or $(1)$.  Such towers we call ``$j$-domino towers", or ``domino towers" if we forget the $j$.  The above tower of $\la=(12,2,1,1)$ is a $2$-domino tower.

Write $\Lambda^D$ for the set of domino towers, and $\Lambda^j$ for the set of $j$-domino towers.  Let $D(n)=|\Lambda^D_n|$ and $D_j(n)=|\Lambda^j_n|$. 
Note that $D(n)=0$ if $n' \equiv 2 \mod 4$, and $D(n)=\sum_{j=1}^{k-1}D_j(n)$ otherwise.

  It is easy to see that
\begin{equation} \label{lute}
\begin{split}
D_j(n) &=2^{2(j-1)}A(n-2^{j+1}) \\
	&= 2^{2(j-1)-k+\binom{k}{2}-\binom{j+1}{2}}A(n) \\
	&=2^{\binom{k}{2}-\binom{j}{2}-k+j-2}A(n), \\
	\end{split}
	\end{equation}
 using Lemma \ref{alpha.increment}.

\begin{cor} 
\cite[Theorem 1]{APS-chiral} 

For $n \geq 2$ and $k=\ord(n')$, the number of chiral partitions is given by
\beq
B(n)=A(n)\left(\half + \epsilon \cdot 2^{\binom{k}{2}-k} +  \sum_{j=1}^{k-1} 2^{\binom{k}{2}-\binom{j}{2}-k+j-2} \right).
\eeq
\end{cor}

The sum is understood to be $0$ if $k=1$.

\begin{proof} The number of partitions satisfying (1) in  Theorem \ref{APS-chiral} is $\half A(n)$, and the number satisfying (3) is
\beq
\epsilon A(n-3)= \epsilon A(n) 2^{\binom{k}{2}-k}. 
\eeq

So 
\beq
\begin{split}
B(n) &=\half A(n)+D(n)+ \epsilon A(n) 2^{\binom{k}{2}-k} \\
	&=\half A(n)+ \left(\sum_{j=1}^{k-1}D_j(n) \right)+ \epsilon A(n) 2^{\binom{k}{2}-k}, \\
	\end{split}
\eeq
and the final formula follows from (\ref{lute}).
\end{proof}

 \subsection{Case of Trivial $2$-core}

Consider the case of partitions of the form $\phi(\alpha,\beta) \vdash 2n$. (See Section \ref{tower.sec}.) Note that $\phi(\alpha,\beta)$ is odd iff $\alpha$ and $\beta$ are odd and $\summ(a,b)$ is neat.  Theorem \ref{APS-chiral} says:

\begin{prop}  The partition $\la=\phi(\alpha,\beta)$ is chiral iff both $\alpha$ and $\beta$ are odd, and one of the following holds:
\begin{enumerate}
\item $\bin(a) \cap \bin(b)=\emptyset$ and $\ord(n) \in \bin(a)$.  
\item $\bin(a) \cap \bin(b)= \{ j \}$, with $j=\ord(a)=\ord(b)$.  
\end{enumerate}
\end{prop}
 
  \bigskip
 
The disjunction of the two conditions of the proposition is equivalent to a certain binomial coefficient being odd.
 
\begin{prop} Let $a+b=n$.  The quantity $\dbinom{n-1}{a-1,b}$ is odd iff one of the following holds:
 \begin{enumerate}
\item $\bin(a) \cap \bin(b)=\emptyset$ and $\ord(n) \in \bin(a)$.   
\item $\bin(a) \cap \bin(b)= \{ j \}$, with $j=\ord(a)=\ord(b)$. 
\end{enumerate}
\end{prop}
 
The proof is left to the reader. $\square$
 
\begin{cor} \label{phi.chiral}  Let $(\alpha,\beta) \models n$ with $|\alpha|=a$ and $|\beta|=b$.  Then $\la=\phi(\alpha,\beta)$ is chiral iff the quantity
\beq
f_\alpha f_\beta \binom{n-1}{a-1,b}
\eeq
is odd.
\end{cor}
 
 \subsection{Tower Merging}

Suppose $\la$ and $\mu$ are  partitions, with the property that for all $i \geq 0$, either the $i$th row of $T_\la$ is empty, or the $i$th row of $T_\mu$ is empty.
In this situation, one can define a partition $\nu$ of $|\la|+|\mu|$ as follows:  For all $i \geq 0$, if the $i$th row of $T_\la$ (resp., $T_\mu$) is nonempty, let the $i$th row of $T_\nu$ be equal to the $i$th row of $T_\la$ (resp. of $T_\mu$).
If the $i$th rows of both $T_\la$ and $T_\mu$ are empty, then let the $i$th row of $T_\nu$ be empty.  We will refer to this process as ``merging" the towers $T_\la$ and $T_\mu$.

For example,  $\la= (3,3,2)$, whose $2$-core tower is
\begin{displaymath}
  \xymatrix@C=0.7em{
    &&&&&&& \emptyset \ar@{-}[dllll]\ar@{-}[drrrr] &&&&&&&\\
    &&&\emptyset\ar@{-}[dll]\ar@{-}[drr] &&&&&&&& \emptyset \ar@{-}[dll]\ar@{-}[drr] &&&\\
    &\emptyset \ar@{-}[dl]\ar@{-}[dr]&&&& (1) \ar@{-}[dl]\ar@{-}[dr] &&&& (1) \ar@{-}[dl]\ar@{-}[dr]&&&& \emptyset \ar@{-}[dl]\ar@{-}[dr]&\\
    \emptyset &&\emptyset&&\emptyset&&\emptyset&&\emptyset&&\emptyset&&\emptyset &&\emptyset \\
  }
\end{displaymath}
 can be merged with $\mu=(3,1,1,1,1,1)$, whose tower is
 \begin{displaymath}
  \xymatrix@C=0.7em{
    &&&&&&& \emptyset \ar@{-}[dllll]\ar@{-}[drrrr] &&&&&&&\\
    &&&\emptyset\ar@{-}[dll]\ar@{-}[drr] &&&&&&&& \emptyset \ar@{-}[dll]\ar@{-}[drr] &&&\\
    &\emptyset \ar@{-}[dl]\ar@{-}[dr]&&&& \emptyset \ar@{-}[dl]\ar@{-}[dr] &&&& \emptyset \ar@{-}[dl]\ar@{-}[dr]&&&& \emptyset \ar@{-}[dl]\ar@{-}[dr]&\\
    \emptyset &&\emptyset&&(1)&&\emptyset&&\emptyset&&\emptyset&&\emptyset &&\emptyset \\
  }
\end{displaymath}
to yield  $\nu=(11,3,2)$, with tower
 \begin{displaymath}
  \xymatrix@C=0.7em{
    &&&&&&& \emptyset \ar@{-}[dllll]\ar@{-}[drrrr] &&&&&&&\\
    &&&\emptyset\ar@{-}[dll]\ar@{-}[drr] &&&&&&&& \emptyset \ar@{-}[dll]\ar@{-}[drr] &&&\\
    &\emptyset \ar@{-}[dl]\ar@{-}[dr]&&&& (1) \ar@{-}[dl]\ar@{-}[dr] &&&& (1) \ar@{-}[dl]\ar@{-}[dr]&&&& \emptyset \ar@{-}[dl]\ar@{-}[dr]&\\
    \emptyset &&\emptyset&&(1)&&\emptyset&&\emptyset&&\emptyset&&\emptyset &&\emptyset .\\
  }
\end{displaymath}

Write `$a+b \domeq n$' when $a+b=n$, and $\bin(a) \cap \bin(b)= \{ j \}$, with $j=\ord(a)=\ord(b)$.  Put $D_{<j}(n)=\sum_{i=1}^{j-1}D_i(n)$.
 
\begin{prop} \label{merge1} Let $a+b=n$ with $k=\ord(n)$.
\begin{enumerate}
\item If $a+b \doteq n$, then $A(n)=A(a)A(b)$.
\item If $a+b \doteq n$ with $k \in \ord(b)$, then $A(a)D(b)=D(n)$.
\item If $a+b \domeq n$ with $\ord(a)=\ord(b)=j$, then $A(a)D(b)=D_{<j}(n)$.
\end{enumerate}
\end{prop}
  
 \begin{proof} 
 In each case a bijective proof is given by the merging of towers:
 In the first case, merging gives a bijection from $\Lambda_a^{\odd} \times \Lambda_b^{\odd}$ to $\Lambda_n^{\odd}$.
 In the second case, it gives a bijection from $\Lambda_a^{\odd} \times \Lambda_b^D$ to $\Lambda_n^D$.
 In the last case, it gives  a bijection from $\Lambda_a^{\odd} \times \Lambda_b^D$ to $\bigcup_{i<j}  \Lambda_n^i$.
 \end{proof}

 \section{Type $B_n$: Hyperoctahedral Groups}
 
 The Coxeter group of type $B_n$, which we denote by $\mb B_n$, is the wreath product $\left(\Z/2\Z \right)\wr S_n$.  It is traditionally called the $n$th hyperoctahedral group, being the symmetries of the standard hyperoctahedron (or standard hypercube) in $\R^n$.  For $a+b=n$, there is a Young subgroup $\mb B_a \times \mb B_b \leq \mb B_n$ in the evident way.
 
The four multiplicative characters of $\mb B_n$ may be described as follows. Write $\eps: \left(\Z/2\Z\right)^n \to \{\pm 1\}$ for the character whose restriction to each factor $\Z/2\Z$ is nontrivial.  As it is $S_n$-invariant, it extends to a multiplicative character $\eps: \mb B_n \to \{\pm 1\}$.  Write $\sgn^0$ for the composition of the projection $\mb B_n \to S_n$ with the sign character of $S_n$.  Finally write $\sgn^1=\eps \cdot \sgn^0$.  The multiplicative characters of $\mb B_n$ are then $1, \eps, \sgn^0$, and $\sgn^1$.  

 \subsection{Irreducible Representations of $\mb B_n$}

Let $\la \vdash n$.  We consider two extensions of the representation $\rho_\lambda$ of $S_n$ to $\mb B_n$, namely
\begin{equation*}
\rho_{\lambda}^{0}(x;w)=\rho_{\lambda}(w) \text{ \ and \ } \rho_{\lambda}^{1}(x;w)=\eps(x)\rho_{\lambda}(w),
\end{equation*}
for $x\in\left(\Z/2\Z\right)^n\text{ and }w\in S_{n}$.   For $(\alpha,\beta)\models n$, define
\begin{equation*}
\rho_{\alpha\beta}=\text{Ind}_{\mathbb{B}_a\times \mathbb{B}_b}^{\mathbb{B}_n} \rho_{\alpha}^{0}\boxtimes\rho_{\beta}^{1}.
\end{equation*}
Then
\begin{equation*}
\Irr(\mb B_n)=\{\rho_{\alpha\beta}\mid(\alpha,\beta)\models n\}.
\end{equation*} 
Put
\begin{equation*}
f_{\alpha \beta}=\dim \rho_{\alpha \beta}=\binom{n}{a,b}f_{\alpha}f_{\beta}.
\end{equation*}

   For a multiplicative character $\omega$ of $\mb B_n$, we define
\beq
\Bip_\omega(n)=\{ (\alpha,\beta) \models n \mid \det \rho_{\alpha \beta}=\omega \},
\eeq
and $N_\omega(n)=|\Bip_\omega(n)|$.  Similarly, for $a+b=n$ we put 
\beq
\Bip_\omega(a,b)=\{ (\alpha,\beta) \models n \mid \alpha \vdash a, \beta \vdash b, \det \rho_{\alpha \beta}=\omega \},
 \eeq
 and $N_\omega(a,b)=|\Bip_\omega(a,b)|$.
 
 \subsection{The Determinant of $\rho_{\alpha\beta}$}
 
 Proposition 29.2 in \cite{Bushnell.Henniart} gives the following formula for the determinant of an induced representation:
 
 \begin{prop} Let $G$ be a finite group, $H \leq G$ a subgroup, and $\rho$ a representation of $H$.  If $\pi=\Ind_H^G \rho$, then
 \beq
 \det \pi=\det(\mb C[G/H])^{\dim \rho} \otimes (\det \rho \circ \ver_{G/H}).
 \eeq
 \end{prop}
 
 Here $\ver_{G/H}: G_{\ab} \to H_{\ab}$ is the ``Verlagerung" map between the abelianizations of $G$ and $H$, and is the main part of our calculation.
 
 Recall the definition of $\ver_{G/H}$:  Let $t: G/H \to G$ be a section of the canonical projection.  Given $g \in G$, for each $x \in G/H$ we have $g t(x)=t(y) h_{x,g}$ for some $y \in G/H$ and $h_{x,g} \in H$.  Note that here $y={}^g x$, with $G/H$ considered as a $G$-set.
 Then 
 \beq
 \ver_{G/H}(g \mod D(G))=\prod_{x \in G/H}^m h_{x,g} \mod D(H).
 \eeq
 
Let us first treat the case where $G=S_n$ and $H$ is the Young subgroup $S_a \times S_b$, with $a+b=n$.
 Then $G_{\ab}=S_n/A_n$, where $A_n$ is the alternating group, and we may identify $H_{\ab}$ with $S_a/A_a \times S_b/A_b$.
 
\begin{prop} Let   $a+b=n$. Then the map 
 \beq
 \ver=\ver_{S_n/(S_a \times S_b)}: S_n/A_n \to S_a/A_a \times S_b/A_b
 \eeq
 is given by
 \beq
 \ver(\tau_n)=\left(\tau_a^{\binom{n-2}{a-2}}, \tau_b^{\binom{n-2}{b-2}} \right).
 \eeq
 Here $\tau_n$ is any transposition in $S_n$, or trivial if $n<2$.
 \end{prop}

 \begin{proof} 
Let  $J_n=\{1,2,\ldots, n\}$, and write `$\wp_a(J_n)$' for the set of subsets of $J_n$ of cardinality $a$.  For instance $J_a \in \wp_a(J_n)$.  The map $S_n \to \wp_a(J_n)$ defined by $g \mapsto {}^g J_a$ descends to an isomorphism of $S_n/(S_a \times S_b)$ with $\wp_a(J_n)$, as $S_n$-sets. Pick any section $t:\wp_a(J_n) \to S_n$; thus ${}^{t(x)}J_a=x$ for all $x \in \wp_a(J_n)$.
 
 For concreteness, let $\tau_n$ be the transposition $s_1=(12)$. 
 Since
 \beq
 h_{x,s_1}=t({}^{s_1} x)^{-1} s_1 t(x),
 \eeq
 we have  $h_{x,s_1} h_{{}^{s_1}x,s_1}=1$.  Thus
 \beq
 \begin{split}
 \ver(\tau_n) &= \prod_{x \in \wp_a(J_n)} h_{x,s_1} \mod A_n\\
 			&= \prod_{x \in \wp_a(J_n) \mid {}^{s_1}x=x} h_{x,s_1} \mod A_n. \\
\end{split}
 \eeq
  
Now ${}^{s_1}x=x$ iff  $\{1,2\}$ is a subset of $x$, or of the complement of $x$.  In the first case, $h_{x,s_1}$ is a transposition $\tau_a \in S_a$, and in the second case, $h_{x,s_1}$ is a transposition $\tau_b \in S_b$.
There are $\binom{n-2}{a-2}$ elements of $\wp_a(J_n)$ containing $\{1,2\}$, and $\binom{n-2}{b-2}$ elements of $\wp_a(J_n)$ whose complement contains $\{1,2\}$.  The result follows.

 \end{proof}

Next we compute the verlagerung for $G=\mb B_n$ and $H$ the Young subgroup $\mb B_a \times \mb B_b$, with $a+b=n$. 
 The derived quotient $(\mb B_n)_{\ab}=\mb B_n/D(\mb B_n)$ is a Klein $4$ group generated by the images of $\tau_n$ and $e_n$, where $e_n \in (\Z/2 \Z)^n$ is any vector with $\eps(e_n)=-1$.
 We identify $H_{\ab}$ with $ (\mb B_a)_{\ab} \times (\mb B_b)_{\ab}$.
 
 \begin{prop} The map
 \beq
 \ver=\ver_{G/H}: (\mb B_n)_{\ab}\to (\mb B_a)_{\ab} \times (\mb B_b)_{\ab}
 \eeq
 is given by
 \beq
 \ver(\tau_n)= \left(\tau_a^{\binom{n-2}{a-2}}, \tau_b^{\binom{n-2}{b-2}} \right)
 \eeq
 and
 \beq
 \ver(e_n)= \left( e_a^{\binom{n-1}{a-1}}, e_b^{\binom{n-1}{b-1}} \right).
 \eeq
 \end{prop}
 
\begin{proof}
We may identify the quotient $\mb B_n/(\mb B_a \times \mb B_b)$ with $S_n/(S_a \times S_b)$, and in particular use a transversal $S_n/(S_a \times S_b) \to S_n$ to form a transversal $t:\mb B_n/(\mb B_a \times \mb B_b) \to \mb B_n$ whose image lies in $S_n$.  The following diagram commutes:

 \[ \begin{tikzcd}
(S_n)_{\ab} \arrow{r}{\ver} \arrow[swap]{d}& (S_a)_{\ab} \times (S_b)_{\ab} \arrow{d} \\%
(\mb B_n)_{\ab} \arrow{r}{\ver}& (\mb B_a)_{\ab} \times (\mb B_b)_{\ab}
\end{tikzcd}
\] 
Here the vertical maps are the obvious inclusions, and the top horizontal map was computed in the previous proposition.  This gives the formula for $\ver(\tau_n)$.
  
Now, with notation as in the previous proof, observe that
\beq
{}^{t(x)^{-1}}x=J_a,
\eeq
so 
\beq
t(x)^{-1}(1) \in J_a \Leftrightarrow 1 \in x.
\eeq  
  We have
  \beq
  h_{x,e_1}=t(x)^{-1}e_1 t(x)=e_i,
  \eeq
  with $i=t(x)^{-1}(1)$.  Therefore, modulo $D(H)$, we may write 
  \beq
  h_{x,e_1} = \begin{cases} e_a, \text{ if $1 \in x$} \\
  e_b, \text{ if $1 \notin x$.}
  \end{cases}
  \eeq
 There are $\binom{n-1}{a-1}$ subsets $x$ of $J_n$ with $1 \in x$, and $\binom{n-1}{b-1}$ with $1 \notin x$, thus the product of the $h_{x,e_1}$ is $\left( e_a^{\binom{n-1}{a-1}}, e_b^{\binom{n-1}{b-1}} \right)$, giving the formula for $\ver(e_n)$.

\end{proof}

 For $i,j$ we have
 \beq
 \det(\rho_\alpha^0(\tau_a^i) \boxtimes \rho_\beta^1(\tau_b^j))=(-1)^{g_\alpha f_\beta i+ g_\beta f_\alpha j},
 \eeq
 and so
 \beq
 \det(\rho_\alpha^0 \boxtimes \rho_\beta^1)(\ver(s_1))=(-1)^{g_\alpha f_\beta \binom{n-2}{a-2}+ g_\beta f_\alpha \binom{n-2}{b-2}}.
 \eeq
 Similarly,
  \beq
 \det(\rho_\alpha^0(e_a^i) \boxtimes \rho_\beta^1(e_b^j))=(-1)^{f_\alpha f_\beta j},
 \eeq
 so
  \beq
 \det(\rho_\alpha^0 \boxtimes \rho_\beta^1)(\ver(e_n))=(-1)^{f_\alpha f_\beta  \binom{n-1}{b-1}}.
 \eeq
 
 The permutation module $\C[\wp_a(J_n)]$ coming from the action of $\mb B_n$ on $\wp_a(J_n)$ factors through the action of $S_n$.  Thus $e_n$ acts trivially, and $s_1$ acts by permuting these sets.  The number of doubleton orbits of $s_1$ on this set is equal to the number of subsets of $\wp_a(J_n)$ which contain $1$ but not $2$, thus equals
 $\binom{n-2}{a-1}$.  This gives
 \beq
 \det(\C[G/H])(e_n)=1 \text{  and } \det(\C[G/H])(s_1)=(-1)^{\binom{n-2}{a-1}},
 \eeq
so that
\beq
\det(\C[G/H])=(\sgn^0)^{\binom{n-2}{a-1,b-1}}.
\eeq

Let 
\beq
x_{\alpha \beta}=f_\alpha f_\beta \binom{n-1}{a,b-1} \in \Z/2 \Z
\eeq
and
\beq
y_{\alpha \beta}=f_\alpha f_\beta \binom{n-2}{a-1,b-1}+ f_\beta g_\alpha \binom{n-2}{a-2,b}+ f_\alpha g_\beta \binom{n-2}{a,b-2} \in \Z/2 \Z.
\eeq

From the above we deduce:

\begin{thm} \label{hyper.Solomon}
 For a bipartition $(\alpha,\beta)$, we have
 \beq
 \det \rho_{\alpha \beta}=\eps^{x_{\alpha \beta}} \cdot (\sgn^0)^{y_{\alpha \beta}}.
 \eeq
 \end{thm}
 
 \begin{proof} The preceding shows that
 \beq
 \det (\rho_{\alpha \beta}(e_n))=(-1)^{x_{\alpha \beta}} \quad \text{and} \quad  \det (\rho_{\alpha \beta}(s_1))=(-1)^{y_{\alpha \beta}},
 \eeq
 and this gives the theorem.
 \end{proof}

An alternate proof of Theorem \ref{hyper.Solomon} can be given through Theorem \ref{solomon.principle2}, via the Frobenius Character Formula to compute $\chi_{\alpha\beta}(s_1)$ and $\chi_{\alpha\beta}(e_1)$.  (Here $\chi_{\alpha\beta}$ denotes the character of $\rho_{\alpha \beta}$.) In fact,
\begin{equation} \label{charles}
\chi_{\alpha\beta}(s_1)={{n-2}\choose{a-2, b}}f_{\beta}\chi_{\alpha}(s_1)+{{n-2}\choose{a, b-2}}f_{\alpha}\chi_{\beta}(s_1)
\end{equation}
and
\beq
\chi_{\alpha\beta}(e_1)=f_{\alpha}f_{\beta}\left(2{{n-1}\choose{a-1, b}}-{{n}\choose{a, b}}\right).
\eeq

Details are left to the reader.

\bigskip
 
 Computing the determinant of $\rho_{\alpha \beta}$ thus reduces to determining the parities of $x_{\alpha \beta}$ and $y_{\alpha \beta}$.
  \begin{thm} \label{x.thm} For $(\alpha,\beta) \models n$, the quantity  $x_{\alpha \beta}$ is odd iff $\la=\phi(\beta,\alpha)$ is chiral.  In other words,
$ g_{\phi(\beta,\alpha)} \equiv x_{\alpha\beta} \mod 2$.
  \end{thm}
\begin{proof} This is immediate from Corollary \ref{phi.chiral}.
\end{proof}

\begin{cor} \label{x.cor} For all $n$,
 \beq
 N_\eps(n)+N_{\sgn^1}(n)=B(2n).
 \eeq
 \end{cor}

An important case is when $\alpha=\beta$.
 
  \begin{prop} \label{hookss} Let $a \geq 2$ and $\alpha \vdash a$.  Then 
 \beq
 \det \rho_{\alpha \alpha} = \begin{cases} \eps, \text{ if $a$ is a power of $2$ and $\alpha$ is a hook} \\
 1, \text{ otherwise.} 
 \end{cases}
 \eeq
 \end{prop}
 
 \begin{proof} 
 In this case 
 \beq
 x_{\alpha \alpha} \equiv f_\alpha \binom{2a-1}{a,a-1}, \text{ and } y_{\alpha \alpha} \equiv 0,
 \eeq
 since $\binom{2a-2}{a-1,a-1}$ is always even.  Now $\bin(a) \cap \bin(a-1) = \emptyset$ iff $a$ is a power of $2$, and in this case the odd partitions are precisely the hooks.
 This gives the proposition.
 \end{proof}

\subsection{Values of $N_\omega(n)$ for small $n$}

Using Theorem \ref{hyper.Solomon}, one can compute:

\bigskip
\begin{center}
 
 \begin{tabular}{|c|c|c|c|c|} 
 \hline
 $n$ & $N_{1}(n)$ & $N_{\sgn^0}(n)$ & $N_{\sgn^1}(n)$ & $N_{\eps}(n)$\\ [0.5ex] 
 \hline\hline
 $1$ & $1$ & $0$ & $0$ & $1$\\
 \hline
 $2$ & $1$ & $1$ & $2$ & $1$\\
 \hline
 $3$ & $2$ & $4$ & $2$ & $2$\\
 \hline
  $4$     & $4$ & $4$ & $8$ & $4$\\
 \hline
 $5$ & $8$ & $20$ & $4$ & $4$\\
 \hline
 $6$ & $33$ & $8$ & $16$ & $8$\\
 \hline 
 $7$ & $46$ & $32$ & $16$ & $16$\\
 \hline
$8$& $69$ & $28$ & $60$ & $28$\\
 \hline
 $9$ & $116$ & $168$ & $8$ & $8$\\
 \hline
$10$     & $417$ & $16$ & $32$ & $16$\\
 \hline
    $11$ & $624$ & $64$ & $32$ & $32$\\
 \hline
 $12$     & $909$ & $64$ & $128$ & $64$\\
 \hline
$13 $    & $1322$ & $320$ & $64$ & $64$\\
 \hline
 $14$ & $2153$ & $128$ & $256$ & $128$\\
 \hline
    $15$ & $2932$ & $512$ & $256$ & $256$\\
  \hline
  $16$ & $4038$ & $424$ & $936$ & $424$\\
 \hline
 \end{tabular}
\end{center}

\bigskip

Below is a logplot, base $2$, of each $N_\omega(n)$ for $2 \leq n \leq 65$.  The horizontal axis is $n$.  The orange line is $N_1(n)$, the green line is $N_{\sgn^0}(n)$, the red line is $N_{\sgn^1}(n)$, and the blue line is $N_{\eps}(n)$.

  \begin{center} \includegraphics[scale=0.7]{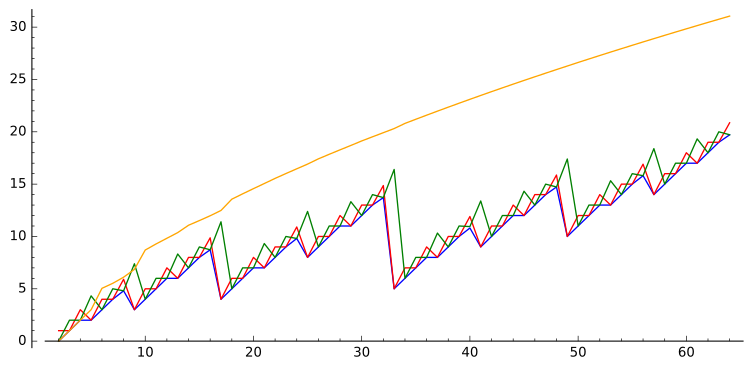}
   \end{center}

  Note the compatibility with Theorem \ref{inequality}.   
   
 \subsection{An Involution of $\Bip(n)$}

\begin{lemma} Mod $2$, we have $g_{\la'}=g_\la+f_\la$, and $f_{\phi(\alpha,\beta)}=f_{\alpha\beta}$.
\end{lemma}

\begin{proof} For the first part, $\chi_{\la'}(s_1)=-\chi_{\la}(s_1)$, so $g_\la+g_{\la'}=f_\la$.  The second part is immediate.
\end{proof}

Given $(\alpha,\beta) \models n$, put $\sigma(\alpha,\beta)=(\alpha',\beta')$.  

\begin{prop} \label{invo.cong} For $(\alpha,\beta) \models n$, we have
\begin{enumerate}
\item $x_{\sigma(\alpha,\beta)}=x_{\alpha\beta}$,
\item $y_{\sigma(\alpha,\beta)} = y_{\alpha\beta}+ f_{\alpha\beta}$.
\end{enumerate}
\end{prop}

 \begin{cor} \label{eps=sgn1}
 Let $n \geq 3$ be odd.  If $a+b=n$, then $\sigma$ restricts to a bijection from $\Bip_{\sgn^1}(a,b)$ to $\Bip_{\eps}(b,a)$.  Thus $N_{\sgn^1}(b,a)=N_{\eps}(a,b)$ and $N_{\sgn^1}(n)=N_{\eps}(n)$. 
 \end{cor}

 \begin{proof} 
 We must show that if $x_{\alpha \beta}=y_{\alpha \beta}=1$, then $x_{\sigma(\alpha,\beta)}=1$ and $y_{\sigma(\alpha,\beta)}=0$.    Now, if $n$ odd and $\binom{n-1}{a,b-1}$ is odd, it is elementary to see that $b$ is odd and $\binom{n}{a,b}$ is odd.  By the previous proposition we deduce that  $y_{\sigma(\alpha,\beta)}=0$.
  \end{proof}
 
 \begin{cor} \label{sgn.eps}  Let $n \geq 3$ be odd.  Then $N_{\sgn^1}(n)=N_{\eps}(n)=\frac{1}{4}A(2n)$.
 \end{cor}
 
 \begin{proof} By Corollaries \ref{x.cor} and \ref{eps=sgn1}, we have $N_{\sgn^1}(n)=N_{\eps}(n)=\half B(2n)$.  But by Corollary \ref{APS-chiral}, $B(2n)=\half A(2n)$.
 \end{proof}

\subsection{Counting with fixed $a,b$.}
 
Given $\omega\in\{1, \eps, \sgn^0, \sgn^1\}$, and $a+b=n$, put
\beq
N_{\omega}(a,b)=\#\{ \alpha \vdash a, \beta \vdash b \mid\det \rho_{\alpha\beta}=\omega\} .
\eeq

We compute $N_\eps(a,b)$ in Proposition \ref{N_eps},  $N_{\sgn^1}(a,b)$ in Proposition \ref{N_sgn1}, and $N_{\sgn^0}(a,b)$ in Proposition \ref{N_sgn0}.  Of course,
\beq
p(a)p(b)=N_1(a,b)+N_{\sgn^0}(a,b)+N_{\sgn^1}(a,b)+N_\eps(a,b),
\eeq
so one implicitly has an expression for $N_1(a,b)$ as well.

Recall that
\beq
x_{\alpha\beta}= f_{\alpha}f_{\beta}{{n-1}\choose{a, b-1}},
\eeq
and 
\beq
y_{\alpha\beta}=f_{\alpha}f_{\beta}{{n-2}\choose{a-1, b-1}}+f_{\beta}g_{\alpha}{{n-2}\choose{a-2, b}}+f_{\alpha}g_{\beta}{{n-2}\choose{a, b-2}}.
\eeq

\begin{lemma} \label{im.at.ccd}
Let $a+b=n$.  The following conditions are equivalent:
\begin{enumerate}
\item ${{n-2}\choose{a-1,b-1}}\text{ is odd, }{{n-2}\choose{a,b-2}}\ \text{is even, and }{{n-2}\choose{a-2,b}}\ \text{is even}$. \\
\item $n$ is even and ${{n-2}\choose{a-1, b-1}}$ is odd.
\end{enumerate}
\end{lemma}

The proof is left to the reader. $\square$

\begin{proposition} \label{N_eps}
Let $a+b=n$. 
If $n$ is even, then
\begin{equation*}
  N_{\eps}(a,b)=
 \begin{cases}
  \frac{1}{2}A(a)A(b) & \text{if } \binom{n-1}{a,b-1} \text{ is odd and } \binom{n-2}{a-1,b-1} \text{ is even} \\
  0 & \text{otherwise.} \\
  \end{cases}
  \end{equation*}
  If $n$ is odd, then
  \begin{equation*}
  N_{\eps}(a,b)=
 \begin{cases}
  \frac{1}{2}A(a)A(b) & \text{if } \binom{n-1}{a,b-1} \text{ is odd}  \\
  0 & \text{otherwise.} \\
  \end{cases}
  \end{equation*}
  \end{proposition}

\begin{proof}
We must count $\alpha \vdash a,\beta \vdash b$ so that $x_{\alpha\beta} \equiv 1$ and $y_{\alpha\beta} \equiv 0$.
For $x_{\alpha\beta}$ to be odd, we need $f_{\alpha}$, $f_{\beta}$ and ${{n-1}\choose{a, b-1}}$ to be odd. 

In this case
\beq
y_{\alpha,\beta}=\binom{n-2}{a-1,b-1}+ g_\alpha \binom{n-2}{a-2,b}+g_\beta \binom{n-2}{a,b-2}.
\eeq
It is now clear that $\Bip_\eps(a,b)=\emptyset$ if either ${{n-1}\choose{a, b-1}}$ is even, or if $n$ is even and ${{n-2}\choose{a-1, b-1}}$ is odd (using Lemma \ref{im.at.ccd}).

So henceforth assume that  ${{n-1}\choose{a, b-1}}$ is odd, and that the conditions of Lemma \ref{im.at.ccd} do not hold.  Since
\beq{}
{{n-1}\choose{a, b-1}}={{n-2}\choose{a-1, b-1}}+{{n-2}\choose{a, b-2}},
\eeq{}
we  consider two cases.

\bigskip

{\bf Case $1$:} ${{n-2}\choose{a-1, b-1}}$ is odd and ${{n-2}\choose{a, b-2}}$ is even.
As we assume the conditions of Lemma \ref{im.at.ccd} do not hold, further ${{n-2}\choose{a-2,b}}$ is odd.
Thus
\beq{}
y_{\alpha\beta}\equiv 1+g_{\alpha} \mod 2.
\eeq{}
So for such $a,b$ we have
\beq
\Bip_{\eps}(a,b)=\{ (\alpha,\beta) \mid f_{\alpha}\text{\ is odd,\ }g_{\alpha}\text{\ is odd}, f_{\beta}\text{\ is odd}\},
\eeq
and by Corollary \ref{f.g.correlation}, we compute
\beq{}
N_{\eps}(a,b)  =\frac{1}{2}A(a)A(b).
\eeq{}
{\bf Case $2$:} If ${{n-2}\choose{a-1, b-1}}$ is even and ${{n-2}\choose{a, b-2}}$ is odd, then
\beq{}
y_{\alpha\beta}\equiv g_{\alpha}{{n-2}\choose{a-2,b}}+g_{\beta} .
\eeq{}
For $y_{\alpha\beta}$ to be even, we need $g_{\alpha}{{n-2}\choose{a-2, b}}$ and $g_{\beta}$ to be both odd or both even.  
We again have two possibilities.\\

{\bf Case $2$a:} If ${{n-2}\choose{a-2, b}}$ is even, then
\beq
\Bip_{\eps}(a,b)=\{ (\alpha,\beta) \mid f_{\alpha}\text{\ is odd,\ } f_{\beta}\text{\ is odd}, g_{\beta}\text{\ is odd}\},
\eeq
so $N_{\eps}(a,b) =\frac{1}{2}A(a)A(b)$, as claimed.

\bigskip

{\bf Case $2$b}: If ${{n-2}\choose{a-2, b}}$ is odd, then
\beq
\begin{split}
\Bip_{\eps}(a,b) &=\{ (\alpha,\beta) \mid f_{\alpha}\text{\ is odd,\ }g_{\alpha}\text{\ is even}, f_{\beta}\text{\ is odd},g_{\beta}\text{\ is even}\} \\
& \cup \{ (\alpha,\beta) \mid f_{\alpha}\text{\ is odd,\ }g_{\alpha}\text{\ is odd}, f_{\beta}\text{\ is odd},g_{\beta}\text{\ is odd}\}.
\end{split}
\eeq
Thus
\beq{}
\begin{split}
N_{\eps}(a,b)&=  A(a)A(b)-A(a)\frac{1}{2}A(b)-\frac{1}{2}A(a)A(b)+2\cdot \frac{1}{2}A(a)\frac{1}{2}A(b)\\
& =\frac{1}{2}A(a)A(b).
\end{split}
\eeq{}
\end{proof}

\begin{proposition}
For given $a$ and $b$, we have
\label{N_sgn1}
\begin{equation}
  N_{\sgn^1}(a,b)=
  \left\{
  \begin{array}{ll}
  0 \ \ \text{if } {{n-1}\choose{a, b-1}}\ \text{is even},\\
  A(a)A(b) \text{ if } n \text{ is even and }{{n-2}\choose{a-1,b-1}}\text{ is odd}, \\
  \frac{1}{2}A(a)A(b) \ \ \text{otherwise.}
  \end{array}
  \right.
  \end{equation}
\end{proposition}
\begin{proof}
For $x_{\alpha\beta}$ to be odd, we need $f_{\alpha}$, $f_{\beta}$ and ${{n-1}\choose{a, b-1}}$ to be odd.  Again there are two cases to consider.
\\
Case $1$: If ${{n-2}\choose{a-1, b-1}}$ is odd and ${{n-2}\choose{a, b-2}}$ is even, then
\beq{}
y_{\alpha\beta}\equiv 1+g_{\alpha}{{n-2}\choose{a-2,b}} \text{\ (mod $2$)}.
\eeq{}
For $y_{\alpha\beta}$ to be odd, we need $g_{\alpha}{{n-2}\choose{a-2, b}}$ to be even.\\
Case $1$a: If ${{n-2}\choose{a-2, b}}$ is even, then
$N_{\sgn^1}(a,b)=A(a)A(b)$.
\\
Case $1$b: If ${{n-2}\choose{a-2, b}}$ is odd, then  
\beq{}
\begin{split}
N_{\sgn^1}(a,b) & = \#\{f_{\alpha}\text{\ is odd,\ }g_{\alpha}\text{\ is even}\}\#\{f_{\beta}\text{\ is odd}\}\\
& = \left[\#\{f_{\alpha}\text{\ is odd}\}-\#\{f_{\alpha}\text{\ is odd,\ }g_{\alpha}\text{\ is odd}\}\right]\#\{f_{\beta}\text{\ is odd}\}\\
& =\frac{1}{2}A(a)A(b).
\end{split}
\eeq{}
Case $2$: If ${{n-2}\choose{a-1, b-1}}$ is even and ${{n-2}\choose{a, b-2}}$ is odd, then
\beq{}
y_{\alpha\beta}\equiv g_{\alpha}{{n-2}\choose{a-2,b}}+g_{\beta} \text{\ (mod $2$)}.
\eeq{}
For $y_{\alpha\beta}$ to be odd, we have two cases:\\
Case $2$a: If $g_{\alpha}{{n-2}\choose{a-2, b}}$ is even and $g_{\beta}$ is odd, then\\
\beq{}
y_{\alpha\beta}\equiv g_{\beta} \text{\ (mod $2$)}.
\eeq{}
Thus, we have
\beq{}
\begin{split}
N_{\sgn^1}(a,b) & = \#\{f_{\beta}\text{\ is odd,\ }g_{\beta}\text{\ is odd}\}\#\{f_{\alpha}\text{\ is odd}\}\\
& =\frac{1}{2}A(b)A(a).
\end{split}
\eeq{}
Case $2$b: If $g_{\alpha}{{n-2}\choose{a-2, b}}$ is odd and $g_{\beta}$ is even, then\\
\beq{}
y_{\alpha\beta}\equiv g_{\beta}+g_{\alpha} \text{\ (mod $2$)}.
\eeq{}
Thus, we have
\beq{}
\begin{split}
N_{\sgn^1}(a,b) & = \#\{f_{\beta}\text{\ is odd,\ }g_{\beta}\text{\ is odd}\}\#\{f_{\alpha}\text{\ is odd,\ }g_{\alpha}\text{\ is even}\}\\
&\ \ \ \ +\#\{f_{\beta}\text{\ is odd,\ }g_{\beta}\text{\ is even}\}\#\{f_{\alpha}\text{\ is odd,\ }g_{\alpha}\text{\ is odd}\}\\
& =\frac{1}{2}A(b)A(a).
\end{split}
\eeq{}

\end{proof}

Our formulas for $N_{\sgn^0}(a,b)$, where $a+b=n$, depend on the parities of the binomial coefficients ${{n-2}\choose{a-2,b}}$, ${{n-2}\choose{a-1,b-1}}$, and ${{n-2}\choose{a,b-2}}$.

\begin{proposition}
 The following table computes $N_{\sgn^0}(a,b)$.
 
\label{N_sgn0}
\vspace{.25cm}
\begin{center}
 \begin{tabular}{|c| c |c||c|} 
 \hline
  ${{n-2}\choose{a-2,b}}$ & ${{n-2}\choose{a-1,b-1}}$ & ${{n-2}\choose{a,b-2}}$ & $N_{\sgn^0}(a,b)$\\ [0.5ex] 
 \hline\hline
 $0$ &  $0$ &  $0$ & $0$\\
 \hline
  $0$ & $1$ &  $0$ & $0$\\
 \hline
 $1$ &  $0$ &  $0$ & $A(b)B(a)$  \\
 \hline
  $0$ & $1$ & $1$  & $A(a)B(b)$  \\
 \hline
 $1$ & $1$ & $1$  & $A(a)B(b)+A(b)B(a)-\half A(a)A(b)$  \\
 \hline
 $1$ & $1$ &  $0$ & $A(b)B(a)-\frac{1}{2}A(a)A(b)$  \\
 \hline
  $0$ &  $0$ & $1$  & $A(a)B(b)-\frac{1}{2}A(a)A(b)$  \\
 \hline
 $1$ &  $0$ & $1$  & $A(a)B(b) + A(b) B(a)- A(a)A(b)$ \\
 \hline
 \end{tabular}
\end{center}

(In the table we read the binomial coefficients mod $2$.)

\vspace{.25cm}
\end{proposition}
\begin{proof}

Each of these eight cases we will refer to by the binary ``code" coming from the parity of the binomial coefficients.  For instance, the first two cases correspond to the codes `$000$' and `$010$'.
Of course, each code further determines the parity of the binomial coefficient  $\binom{n-1}{a,b-1}$.

For the first two cases, suppose that ${{n-2}\choose{a-2,b}}$ and ${{n-2}\choose{a,b-2}}$ are even, but  $x_{\alpha\beta}=0$ and  
\beq
y_{\alpha \beta}=f_\alpha f_\beta {{n-2}\choose{a-1,b-1}} =1.
\eeq
Then $f_\alpha f_\beta=1$, so it must be that $\binom{n-1}{a,b-1}=0$.  But then necessarily $ {{n-2}\choose{a-1,b-1}}=0$, a contradiction.
Thus $\Bip_{\sgn^0}(a,b)=\emptyset$ in the `$000$' and `$010$' cases.

For the `$100$' case, we automatically have $x_{\alpha \beta}=0$ and $y_{\alpha\beta}= f_{\beta}g_{\alpha}$.  Thus $(\alpha,\beta) \in \Bip_{\sgn^0}(a,b)$ iff $f_\beta=g_\alpha=1$.
So $N_{\sgn^0}(a,b)=A(b)B(a)$ here.  

For the `$011$' case, again $x_{\alpha \beta}=0$ but
\beq
y_{\alpha \beta}=f_\alpha(f_\beta+g_\alpha),
\eeq
which is odd iff $f_\alpha$ is odd and $f_\beta \neq g_\alpha$. By Corollary \ref{f.g.correlation}, this gives $N_{\sgn^0}(a,b)=A(a)B(b)$.
 
Now consider the `$111$' case, again with  $x_{\alpha \beta}=0$, but
\beq
y_{\alpha\beta}=f_{\alpha}f_{\beta} +f_{\beta}g_{\alpha} +f_{\alpha}g_{\beta}
\eeq
 must be odd.  Note that $f_\alpha$ or $g_\alpha$ must be odd; we consider the three possibilities.
 For $f_\alpha=g_\alpha=1$, then $y_{\alpha \beta}=g_\beta$; this gives $\half A(a)B(b)$.
 For $f_\alpha=0$, $g_\alpha=1$, then $y_{\alpha \beta}=f_\beta$; this gives $(B(a)-\half A(a))A(b)$.
Next $f_\alpha=1$, $g_\alpha=0$ gives $y_{\alpha \beta}=f_\beta+g_\beta$; this gives $\half A(a)B(b)$.
The sum of these three quantities is equal to $A(a)B(b)+A(b)B(a)-\half A(a)A(b)$.

For the `$110$' case, we must have $f_\alpha f_\beta=0$ and so
 \beq
y_{\alpha\beta}= f_{\beta}g_{\alpha}. 
\eeq
So $(\alpha,\beta) \in \Bip_{\sgn^0}$ iff  $f_\alpha=0$ and $f_\beta=g_\alpha=1$.  This gives 
\beq
\left(B(a)-\half A(a) \right)A(b)
\eeq
 possibilities.

For both `$001$', we must have $f_\beta=0$, and $f_\alpha=g_\beta=1$, giving 
\beq
A(a)\left(B(b)-\half A(b)\right).
\eeq

For `$101$', to solve
\beq
y_{\alpha \beta}=f_\beta g_\alpha + f_\alpha g_\beta =1
\eeq
with $f_\alpha f_\beta=0$, there are two possibilities.  
Either $f_\alpha=0, f_\beta=g_\alpha=1$, or  $f_\beta=0, f_\alpha=g_\beta=1$.
The first possibility gives $(B(a)-\half A(a))A(b)$, and the second gives $(B(b)-\half A(b))A(a)$.  Adding gives the result.

  \end{proof}
 
 \subsection{Final Count}

 We compute $N_\eps(n)$ in Theorem \ref{N_eps(n)}, $N_{\sgn^0}(n)$ in Theorem \ref{N_sgn0(n)}, and $N_{\sgn^1}(n)$ in Theorem \ref{N_sgn1(n)}.  The quantity $N_1(n)$ can be computed in principle from the formula
 \beq
p_2(n)=N_1(n)+N_{\sgn^0}(n)+N_{\sgn^1}(n)+N_\eps(n),
\eeq
 where $p_2(n)$ denotes the number of bipartitions of $n$.

 \begin{theorem}
\label{N_eps(n)}
For $n \geq 2$, and $k=\ord(n)$, we have
\begin{equation*}
N_{\eps}(n)= \begin{cases}
    \frac{1}{4} A(2n)   &\text{if $n$ is odd}, \\
 \frac{1}{8}A(2n)  \left(2+\sum_{j=1}^{k-1}2^{\binom{k}{2}-\binom{j}{2}} \right)  & \text{if $n$ is even}.\\
  \end{cases}
\end{equation*}
\end{theorem}
 
 Note that the sum is $0$ when $k=1$, and similarly for other sums in this paper.
 
\begin{proof}
If $n$ is odd, this follows from Corollary \ref{sgn.eps}, so let $n$ be even.
 
By Proposition \ref{N_eps}, we have
\beq
N_\eps(n) = \sum \limits_{\substack{a+(b-1) \doteq n-1 \\ \summ(a-1,b-1) \messy}} \half A(a)A(b).
\eeq
If $a$ and $b$ are odd, then $\summ(a-1,b-1)$ is neat iff $\summ(a,b-1)$ is neat, so $N_{\eps}(a,b)=0$.   If $a$ and $b$ are even, then $\summ(a-1,b-1)$ is automatically messy.  Hence,
\begin{equation} \label{value}
\begin{split}
N_{\eps}(n) & =\sum_{\substack{a+(b-1) \doteq n-1 \\ a,b \text{ even}}}\frac{1}{2}A(a)A(b)\\
&=\sum_{\substack{a+(b-1) \doteq n-1 \\ a,b \text{ even}}} \frac{1}{2}A(a)A(b-1)2^{\ord(b)-{{\ord(b)}\choose{2}}}\\
&=\frac{1}{2}A(n-1)\sum_{\substack{a+(b-1) \doteq n-1 \\ a,b \text{ even}}}2^{\ord(b)-{{\ord(b)}\choose{2}}}.\\
\end{split}
\end{equation}

Now   $\summ(a,b-1)$ is neat iff  $\ord(a) \geq \ord(b)$ and $\summ(a,b-2^{\ord(b)})$ is neat.   There are two cases to consider.  First if $\ord(b)=k$, then we must also have $k<\ord(a)$.  The sum of these terms is
\beq
   \frac{1}{2}A(n-1) 2^{\nu(n-2^k)+ k- \binom{k}{2}}.
   \eeq
  The other case has $\ord(a)=\ord(b)<k$, the sum over these $a,b$ is
 \beq
  \frac{1}{2}A(n-1) \sum_{j=1}^{k-1} 2^{\nu(n-2^{j+1})+ j -\binom{j}{2}}.
  \eeq
   
   Further simplification comes by using (\ref{nu.formula}):
   \beq 
N_{\eps}(n)= \left(   \half A(n-1)2^{\nu(n)+k-2} \sum_{j=1}^{k-1} 2^{- \binom{j}{2}} \right) + \half A(n-1)2^{\nu(n)+k-1-\binom{k}{2}}.
\eeq
  Now 
  \beq
  A(n-1)=A(n) 2^{-k+\binom{k}{2}},
  \eeq
  so we may rewrite this as
     \beq 
     \begin{split}
N_{\eps}(n) &=  \left(   \half A(n) 2^{\nu(n)-2}  \sum_{j=1}^{k-1} 2^{\binom{k}{2}- \binom{j}{2}} \right) + \half A(n)2^{\nu(n)-1} \\
			&=   \left(   \frac{1}{8} A(2n) \sum_{j=1}^{k-1} 2^{\binom{k}{2}- \binom{j}{2}} \right)             +             \frac{1}{4}A(2n). \\
			\end{split}
\eeq

If $k=1$, then the sum is empty, giving $N_{\eps}(n)= \frac{1}{4}A(2n)$.

\end{proof}

\begin{theorem}
The following computes $N_{\sgn^0}(n)$ in all cases.  Let $k=\ord(n')$.
\label{N_sgn0(n)}
\begin{enumerate}
    \item If  $n \equiv 1 \mod 4$, then
    \beq
N_{\sgn^0}(n) =\frac{1}{4} A(2n) \left( 1+3 \cdot 2^{\binom{k}{2}-1}+2^{\binom{k}{2}-k+1}+ \sum_{j=2}^{k-1} \left(2^{\binom{k}{2}-\binom{j}{2}}+2^{\binom{k}{2}-j} \right)\right).
\eeq    
    \item If  $n \equiv 3 \mod 4$, then   $ N_{\sgn^0}(n) =  \half A(2n)$. 
    \item If  $n$ is even, then
    \beq
N_{\sgn^0}(n) =     \frac{1}{8}A(2n)  \left(2+\sum_{j=1}^{k-1}2^{\binom{k}{2}-\binom{j}{2}} \right).
 \eeq
        
\end{enumerate}
\end{theorem}

\begin{proof}

{\bf Case $n \equiv 2 \mod 4$:}  The codes `$011$', `$111$', `$101$', and `$110$' cannot occur in the $(n-2)$nd row of Pascal's triangle mod $2$. Moreover for $a,b$ corresponding to the codes `$000$' and `$010$' we have
$N_{\sgn^0}(a,b)=0$.  So we are left with the codes  `$100$' and `$001$'.

Let us write 
 
\beq
N_{\sgn^0}(n)^{100}=\sum \limits_{\substack{\summ(a-2,b) \text{ neat} \\ \summ(a-1,b-1) \text{ messy} \\ \summ(a,b-2) \text{ messy}}} N_{\sgn^0}(a,b),
\eeq
and similarly for other codes.  From Proposition \ref{N_sgn0}, $N_{\sgn^0}(a,b)=A(b)B(a)$. 
In this case, since $n-2 \equiv 0 \mod 4$, the first neatness condition on $a-2,b$ implies the other two.  Moreover it implies that $a \equiv 2 \mod 4$, so that $B(a)=\half A(a)$ and $A(a-2)=\half A(a)$.
It follows that
\beq
\begin{split}
N_{\sgn^0}(n)^{100} &= \sum_{(a-2)+b \doteq n-2}\half A(a)A(b) \\
			&=  \sum_{(a-2)+b \doteq n-2} A(a-2)A(b) \\
			&=2^{\nu(n-2)}A(n-2) \\
			&=\frac{1}{4}A(2n). \\
			\end{split}
			\eeq
Next, the condition that $\summ(a,b-2)$ is neat implies that $b \equiv 2 \mod 4$, so $B(b)=\half A(b)$.  Thus

 \beq
\begin{split}
N_{\sgn^0}(n)^{001} &= \sum_{a+(b-2) \doteq n-2} A(a)B(b)-\half A(a)A(b) \\
			&=0.\\
			\end{split}
			\eeq

We conclude that $N_{\sgn^0}(n)=\frac{1}{4} A(2n)$.

\bigskip

{\bf Case $n \equiv 3 \mod 4$:}  This time, the codes `$111$' and `$101$' do not occur in the $(n-2)$nd row of Pascal's triangle mod $2$.  It follows that
\beq
N_{\sgn^0}(n)=N_{\sgn^0}(n)^{100}+N_{\sgn^0}(n)^{011}+N_{\sgn^0}(n)^{110}+N_{\sgn^0}(n)^{001}.
\eeq
Now
\begin{equation} \label{100.001}
N_{\sgn^0}(n)^{100}+N_{\sgn^0}(n)^{001}=\sum \limits_{\substack{\summ(a-2,b) \text{ neat} \\ \summ(a-1,b-1) \text{ messy} \\ \summ(a,b-2) \text{ messy}}} 2A(b)B(a)-\half A(a)A(b),
\end{equation}
and
\begin{equation} \label{011.110}
N_{\sgn^0}(n)^{011}+N_{\sgn^0}(n)^{110}=\sum \limits_{\substack{\summ(a-2,b) \text{ messy} \\ \summ(a-1,b-1) \text{ neat} \\ \summ(a,b-2) \text{ neat}}} 2A(a)B(b)-\half A(a)A(b).
\end{equation}

Next, the conditions on $a$ and $b$ for the sum in (\ref{100.001}) are equivalent to $(a-2)+b \doteq n-2$, $a \equiv 3 \mod 4$, and $b \equiv 0 \mod 4$.  But then $B(a)=A(a)$, and $A(a)=2A(a-2)$.  So (\ref{100.001}) becomes
\beq
3 \cdot 2^{\nu(n)-3}A(n).
\eeq

Meanwhile, the conditions in (\ref{011.110}) are equivalent to $(a-1)+(b-1) \doteq n-2$, $a \equiv 1 \mod 4$, and $b \equiv 2 \mod 4$.  Now $B(b)=\half A(b)$ and $A(b)=2A(b-2)$.  So  (\ref{011.110}) becomes
\beq
2^{\nu(n)-3}A(n).
\eeq

The result follows, since $A(n)2^{\nu(n)}=A(2n)$.

\bigskip

{\bf Case $n \equiv 1 \mod 4$:}
This time we have
\beq
N_{\sgn^0}(n)=N_{\sgn^0}(n)^{100}+N_{\sgn^0}(n)^{011}+N_{\sgn^0}(n)^{111}+N_{\sgn^0}(n)^{110}+N_{\sgn^0}(n)^{001}.
\eeq
Again, Equations (\ref{100.001}) and (\ref{011.110}) hold, but we must add
\beq
N_{\sgn^0}(n)^{111} = \sum \limits_{\substack{\summ(a-2,b) \text{ neat} \\ \summ(a-1,b-1) \text{ neat} \\ \summ(a,b-2) \text{ neat}}} A(a)B(b)+A(b)B(a)-\half A(a)A(b).
\eeq

Now 
\beq
N_{\sgn^0}(n)^{111}=N_{\sgn^0}(n)_{0}^{111}+N_{\sgn^0}(n)_{1}^{111}+N_{\sgn^0}(n)_{2}^{111}+N_{\sgn^0}(n)_{3}^{111},
\eeq
 where $N_{\sgn^0}(n)_{i}^{111}$ takes the sum over $a \equiv i \mod 4$.  But since $N_{\sgn^0}(a,b)=N_{\sgn^0}(b,a)$ for such $a,b$, we also have  $N_{\sgn^0}(n)_3^{111}=N_{\sgn^0}(n)_2^{111}$ and  $N_{\sgn^0}(n)_0^{111}=N_{\sgn^0}(n)_1^{111}$.
 
We have
\beq
\begin{split}
 N_{\sgn^0}(n)_2^{111}& = \sum \limits_{\substack{\summ(a-2,b) \text{ neat} \\ a \equiv 2 \mod 4 }} A(a)A(b) \\
 &= 2A(n-2) 2^{\nu(n-5)} \\
 &= 2^{\binom{k}{2}-3}A(2n),
 \end{split}
 \eeq
 since $B(a)=\half A(a)$, $B(b)=A(b)$, and $A(a)=2A(a-2)$.

On the other hand, putting $k=\ord(n')$,
\beq
N_{\sgn^0}(n)_0^{111}=\sum_{j=2}^{k-1} \sum \limits_{\substack{(a-2)+b \doteq n\\  \ord(a)=j}} A(a)B(b)+A(b)B(a)-\half A(a)A(b).
\eeq

Note that here, $\ord(a)=\ord(b-1)=j$.
So $N_{\sgn^0}(n)_0^{111}$ equals
\beq
\sum_{j=2}^{k-1} \sum \limits_{\substack{(a-2)+b \doteq n\\  \ord(a)=j}} A(a)A(b)2^{\binom{j}{2}-j}+A(a)D(b)+\half A(a)A(b)+A(b)D(a).
\eeq
Now $a+b \domeq n$ with $\ord(a)=\ord(b)=j$, so by Proposition \ref{merge1}, $A(a)D(b)=A(b)D(a)=D_{<j}(n)$, so this gives
\beq
\sum_{j=2}^{k-1} \sum \limits_{\substack{(a-2)+b \doteq n-2\\  \ord(a)=j}} A(a)A(b)\left(\half + 2^{\binom{j}{2}-j} \right)+2D_{<j}(n).
\eeq
Since $A(a-2)=2^{-j+\binom{j}{2}}A(a)$, we get
\beq
\sum_{j=2}^{k-1} \sum \limits_{\substack{(a-2)+b \doteq n-2\\  \ord(a)=j}} A(n-2)2^{j-\binom{j}{2}}\left(\half + 2^{\binom{j}{2}-j} \right)+2D_{<j}(n),
\eeq
which equals
\beq
\sum_{j=2}^{k-1} 2^{\nu(n)+k-j-3} \left[ A(n-2)2^{j-\binom{j}{2}}\left(\half + 2^{\binom{j}{2}-j} \right)+2D_{<j}(n) \right].
\eeq
and then
\beq
A(n) 2^{-k+\binom{k}{2}} \sum_{j=2}^{k-1} 2^{\nu(n)+k-j-3} \left[  2^{j-\binom{j}{2}}\left(\half + 2^{\binom{j}{2}-j} \right)\right] + \sum_{j=2}^{k-1}2^{\nu(n)+k-j-2}     D_{<j}(n).
\eeq
The first sum here is
\beq
A(2n) 2^{\binom{k}{2}-3} \sum_{j=2}^{k-1} \left( 2^{-\binom{j}{2}-1}+2^{-j}  \right),
\eeq
which equals $0$ if $k <3$. 

The second sum is
\beq
\begin{split}
2^{\nu(n)-2} \sum_{j=2}^{k-1} 2^{k-j}D_{<j}(n) &= 2^{\nu(n)-1} \sum_{i=1}^{k-1}(2^{k-i-1}-1) D_i(n) \\
		&= \frac{1}{16}A(2n)  \left(\sum_{j=1}^{k-1} 2^{\binom{k}{2}-\binom{j}{2}} \right)-2^{\nu(n)-1}D(n). 
		\end{split}
\eeq

Therefore if $k \geq 3$ we have

\beq
N_{\sgn^0}(n)^{111}_0=A(2n) 2^{\binom{k}{2}} \sum_{j=2}^{k-1} \left(2^{-\binom{j}{2}-4}+2^{-j-3} \right) +\frac{1}{16}A(2n)  \left(\sum_{j=1}^{k-1} 2^{\binom{k}{2}-\binom{j}{2}} \right)-2^{\nu(n)-1}D(n);
\eeq
note that it is $0$ if $k=2$.

From before,
\beq
 N_{\sgn^0}(n)^{111} =2N_{\sgn^0}(n)^{111}_0+2N_{\sgn^0}(n)^{111}_2.
 \eeq
 If $k=2$, then 
 \beq
 N_{\sgn^0}(n)^{111} =\half A(2n).
 \eeq

If $k \geq 3$, then $N_{\sgn^0}(n)^{111}$ equals
\beq
2^{\binom{k}{2}-2}A(2n)+A(2n) 2^{\binom{k}{2}} \sum_{j=2}^{k-1} \left(2^{-\binom{j}{2}-3}+2^{-j-2} \right) +\frac{1}{8}A(2n)  \left(\sum_{j=1}^{k-1} 2^{\binom{k}{2}-\binom{j}{2}} \right)-2^{\nu(n)}D(n).
\eeq

Meanwhile,
\beq
N_{\sgn^0}(n)^{011+110}=\sum \limits_{\substack{\summ(a-2,b) \text{ messy} \\ \summ(a-1,b-1) \text{ neat} \\ \summ(a,b-2) \text{ neat}}} 2A(a)B(b)-\half A(a)A(b).
\eeq
The conditions on $a,b$ imply that $\summ(a,b)$ is neat, $a \equiv 1 \mod 4$, $b \equiv 0 \mod 4$, and $k \in \bin(b)$.  So
\beq
\begin{split}
2A(a)B(b)-\half A(a)A(b) &=2A(a)(D(b)+\half A(b))-\half A(a)A(b) \\
&=2D(n)+\half A(n). \\
\end{split}
\eeq
The number of $a,b$ satisfying this condition is $2^{\nu(n)-2} $, so we have
\beq
N_{\sgn^0}(n)^{011+110}=2^{\nu(n)-2} \left(2D(n)+\half A(n) \right).
\eeq

Next,

 \beq 
N_{\sgn^0}(n)^{100+001} =\sum \limits_{\substack{\summ(a-2,b) \text{ neat} \\ \summ(a-1,b-1) \text{ messy} \\ \summ(a,b-2) \text{ messy}}} 2A(b)B(a)-\half A(a)A(b).
\eeq
 The conditions on $a,b$ again imply that $\summ(a,b)$ is neat, $a \equiv 1 \mod 4$, $b \equiv 0 \mod 4$, and $k \in \bin(a)$.  So
 \beq
 \begin{split}
 2A(b)B(a)-\half A(a)A(b) &=\half A(n)+2^{\binom{k}{2}-k+1}A(n)+2D(n).
 \end{split}
 \eeq
 The number of $a,b$ satisfying this condition is again $2^{\nu(n)-2} $, so we have
\beq
N_{\sgn^0}(n)^{100+001}=2^{\nu(n)-2} \left(\half A(n)+2^{\binom{k}{2}-k+1}A(n)+2D(n) \right).
 \eeq

Thus
\beq
\begin{split}
N_{\sgn^0}(n)^{011+110+100+001} &=2^{\nu(n)-2} \left(A(n)+2^{\binom{k}{2}-k+1}A(n)+4D(n) \right) \\
\end{split}
\eeq

We obtain $N_{\sgn^0}(n)=\frac{5}{4}A(2n)$, if $k=2$.  For $k \geq 3$ we finally obtain
\beq
N_{\sgn^0}(n) =A(2n) \left( \frac{1}{4}+3 \cdot 2^{\binom{k}{2}-3}+2^{\binom{k}{2}-k-1}+ \sum_{j=2}^{k-1} \left(2^{\binom{k}{2}-\binom{j}{2}-2}+2^{\binom{k}{2}-j-2} \right)\right).
\eeq

{\bf Case $n \equiv 0 \mod 4$:}
Now
 \beq
 N_{\sgn^0}(n)= N_{\sgn^0}(n)^{100}+ N_{\sgn^0}(n)^{001}+ N_{\sgn^0}(n)^{101}.
 \eeq
We have
 \beq
 N_{\sgn^0}(n)^{100}=\sum\limits_{\substack{(a-2)+b \doteq n-2 \\   \summ(a,b-2)\text{ messy}}} A(b)B(a),
 \eeq
 
 \beq
 N_{\sgn^0}(n)^{001}=\sum\limits_{\substack{a+(b-2) \doteq n-2 \\   \summ(a-2,b)\text{ messy}}} A(a)D(b).
 \eeq
 
 Since $B(n)=\half A(n)+D(n)$ for $n$ even, we may write
 
 \beq
 N_{\sgn^0}(n)^{101}=\sum\limits_{\substack{(a-2)+b \doteq n-2 \\   a+(b-2) \doteq n-2}} A(a)D(b)+A(b)D(a).
 \eeq

  So
 \beq
 N_{\sgn^0}(n)^{100}=\sum\limits_{\substack{a+b \doteq n \\   k \in \bin(a) }} A(b) B(a),
 \eeq
 and
  \beq
 N_{\sgn^0}(n)^{001}=\sum\limits_{\substack{a+b \doteq n \\   k \in \bin(b) }}A(a)D(b).
 \eeq
 Thus
 \begin{equation} \label{peter}
 \begin{split}
  N_{\sgn^0}(n)^{100}+ N_{\sgn^0}(n)^{001} &= \sum\limits_{\substack{a+b \doteq n \\   k \in \bin(b) }}A(a) \left(2D(b)+\half A(b) \right) \\
  &=  \sum\limits_{\substack{a+b \doteq n \\   k \in \bin(b) }} 2D(n)+ \half A(n) \\
  &=  2^{\nu(n)} \left( D(n)+ \frac{1}{4} A(n) \right).\\
  \end{split}
  \end{equation}
  We have used Proposition \ref{merge1}  for  the equality $A(a)D(b)=D(n)$.

  \begin{equation} \label{donuts}
  \begin{split}
 N_{\sgn^0}(n)^{101}&=2\sum\limits_{\substack{(a-2)+b \doteq n-2 \\   a+(b-2) \doteq n-2}} A(a)D(b) \\
 				&=2 \sum_{j=2}^{k-1}\sum\limits_{\substack{a+b \domeq n \\   \ord(b)=j}} A(a)D(b) \\ 
				&=2 \sum_{j=2}^{k-1} 2^{\nu(n-2^{j+1})} D_{< j}(n) \\ 
				&=2  \sum_{j=2}^{k-1} 2^{\nu(n)+k-j-2} D_{< j}(n) \\
				&= 2^{\nu(n)-1}  \sum_{j=2}^{k-1} 2^{k-j} D_{< j}(n) \\
				&=2^{\nu(n)} \sum_{i=1}^{k-1} (1+2+ \cdots + 2^{k-i-2})D_i(n).
				\end{split}
 \end{equation}
 
 We have again used (\ref{nu.formula}).  Therefore
\beq
\begin{split}
   N_{\sgn^0}(n) &= 2^{\nu(n)} \left(\sum_{j=1}^{k-1} (1+2+ \cdots + 2^{k-j-2})D_j(n) + \left(\sum_{j=1}^{k-1} D_j(n) \right) + \frac{1}{4}A(n) \right) \\
   &=2^{\nu(n)}\left(\frac{1}{4}A(n) + \sum_{j=1}^{k-1}2^{k-j-1}D_j(n) \right) \\
   &= 2^{\nu(n)}A(n) \left(\frac{1}{4} + \sum_{j=1}^{k-1}2^{\binom{k}{2}-\binom{j}{2}-3} \right)\\
   &=\frac{1}{8}A(2n) \left(2+  \sum_{j=1}^{k-1}2^{\binom{k}{2}-\binom{j}{2}} \right). \\
   \end{split}
   \eeq

\end{proof}

\begin{cor} \label{eps=sgn^0} When $n$ is even, we have $N_\eps(n)=N_{\sgn^0}(n)$.
\end{cor}
\qedsymbol
\bigskip

 \begin{theorem}
\label{N_sgn1(n)}
For $n \geq 2$, and $k=\ord(n)$, we have
\begin{equation*}
N_{\sgn^1}(n)= \begin{cases}
\frac{1}{4}A(2n) &\text{if } n \text{ is odd} \\
\frac{1}{8}A(2n) \left( 2+2^k+ \sum_{j=1}^{k-1} 2^{\binom{k}{2}-\binom{j}{2}}(2^j-1) \right) & \text{if $n$ is even}.\\
  \end{cases}
\end{equation*}
\end{theorem}

\begin{proof} For $n$ odd, this is Corollary \ref{sgn.eps}.  Otherwise this follows from the identity $N_{\eps}(n)+N_{\sgn^1}(n)=B(2n)$, and the formulas for $N_{\eps}(n)$ and $B(2n)$.
\end{proof}

\subsection{Inequalities}
 In this section, we discuss how to establish, for $n \geq 10$:
 \begin{enumerate}
\item  $N_{\eps}(n)=N_{\sgn^1}(n)<N_{\sgn^0}(n)<N_{1}(n)$, for $n$ odd
\item $N_{\eps}(n)=N_{\sgn^0}(n)<N_{\sgn^1}(n)<N_{1}(n)$, for $n$ even.
\end{enumerate}
 from the formulas in the previous section.  We have already established $N_{\eps}(n)=N_{\sgn^1}(n)$ when $n$ is odd in Corollary \ref{eps=sgn1}, and $N_{\eps}(n)=N_{\sgn^0}(n)$ when $n$ is even in Corollary \ref{eps=sgn^0}.  

If $n$ is even, we have
 \beq
 \begin{split}
 N_{\eps}(n) &=\frac{1}{8}A(2n) \left( 2+\sum_{j=1}^{k-1}2^{\binom{k}{2}-\binom{j}{2}} \right) \\
&< \frac{1}{8}A(2n) \left( 2+2^k+ \sum_{j=1}^{k-1} 2^{\binom{k}{2}-\binom{j}{2}}(2^j-1) \right) \\
&=N_{\sgn^1}(n). \\
\end{split}
\eeq
  If $n$ is odd, then $N_{\eps}(n)=\frac{1}{4}A(2n)$, and it is obvious from the formulas that $N_{\eps}(n)<N_{\sgn^0}(n)$.
 
 Comparing these formulas to 
 \beq
 N_1(n)=p_2(n)-N_{\eps}(n)-N_{\sgn^0}(n)-N_{\sgn^1}(n)
 \eeq
  is considerably more difficult, since the partition function itself does not have an explicit formula.   Let us sketch how this can be done.  We will take $n$ even for our sketch; the case $n$ odd is similar.  By the earlier inequalities, we have
 \beq
 \begin{split}
 \frac{N_1(n)}{N_{\sgn^1}(n)} &> \frac{p_2(n)}{N_{\sgn^1}(n)}-3 \\
& >\frac{p_2(n)}{\frac{1}{8}A(2n)(2+2^{k+1})}-3. \\
 \end{split}
 \eeq
 So we must estimate $p_2(n)$ from below and $A(2n)$ from above; note that $2^k \leq n$.  It is easy to see that $p_2(n) \geq 2p(n)$.
 Lemma \ref{A(n).bound} yields the (subexponential) upper bound 
\beq
A(2n) \leq 2n^{\half (\log_2n+3)}.
\eeq
  According to Hardy-Ramanujan \cite{Hardy.Ram}, we have 
\beq
p(n) \sim \frac{1}{4n \sqrt{3}} \exp \left( \pi \sqrt{ \frac{2n}{3}} \right)
\eeq
as $n \to \infty$.  From the above it is straightforward to compute that
\beq
\lim_{n \to \infty} \frac{N_1(n)}{N_{\sgn^0}(n)}=\infty.
\eeq
But the inequality for all $n \geq 10$ takes more work.  For this, one may start with the classical estimate 
\beq
p(n) \geq \frac{n^{j-1}}{j! (j-1)!}
\eeq
for all $1\leq j \leq n$ (see \cite{Igor}), and use the estimate
\beq
k! \leq \sqrt{2 \pi} k^{k+\half} e^{-k} e^{\frac{1}{12k}}
\eeq
 of H. Robbins in \cite{Robbins}.
With care one can prove from these estimates, and from computation for small $n$, that
\beq
p(n) \geq 5nA(n)
\eeq
for $n \geq 64$.  Our desired inequality ultimately follows from this.

This concludes our sketch of how to prove (for $n \geq 10$) that $N_{\sgn^1}(n)<N_{1}(n)$ for $n$ even, and similarly $N_{\sgn^0}(n)<N_{1}(n)$ for $n$ odd.

 \section{Type $D_n$: Demihyperoctahedral Groups}
We denote by $\mb D_n$ the kernel of $\eps: \mb B_n \to \{ \pm 1\}$; it is the Weyl group of type $D_n$, called the demihyperoctahedral group.
Since $\mb D_1$ is trivial, $\mb D_2$ is the Klein $4$-group, and $\mb D_3$ is isomorphic to $S_4$, we may take $n \geq 4$ when convenient.

The group $\mb D_n$, for $n \geq 2$, has two multiplicative characters: $1$ and $\sgn$, where $\sgn$ is the restriction of $\sgn^0$ (or of $\sgn^1$) to $\mb D_n$.  To avoid confusion with earlier notation, let us write
\beq
N_\omega'(n)= \# \{ \pi \in \Irr(\mb D_n) \mid \det \pi=\omega \},
\eeq
for $\omega=1,\sgn$.

\subsection{Clifford Theory} In this subsection we assume Clifford Theory (especially the index $2$ case) as well-known, and tacitly apply it to the subgroup $\mathbb{D}_n$ of $\mathbb{B}_n$.  

Firstly, the restriction of $\rho_{\alpha,\beta}$ from $\mb B_n$ to $\mathbb{D}_n$ is irreducible iff $\alpha \neq \beta$.  Let us call representations obtained this way ``Type I" representations.  Moreover $\Res^{\mathbb{B}_n}_{\mathbb{D}_n} \rho_{\alpha_1,\beta_1} \cong \Res^{\mathbb{B}_n}_{\mathbb{D}_n} \rho_{\alpha_2,\beta_2}$ iff $(\alpha_2,\beta_2)$ equals $(\alpha_1,\beta_1)$ or $(\beta_1,\alpha_1)$.

Next, if $n$ is even and $\alpha \vdash \dfrac{n}{2}$, then the restriction of $\rho_{\alpha,\alpha}$ to $\mathbb{D}_n$ is a direct sum of two non-isomorphic irreducible representations, say $\rho_{\alpha}^+$   and $\rho_{\alpha}^-$.  Let us call these``Type II" representations.   The representation $\rho_{\alpha}^-$ is the twist of $\rho_{\alpha}^+$ by $e_1 \in \mathbb{B}_n - \mathbb{D}_n$.  We will write $\chi_\alpha^{\pm}$ for the character of $\rho_\alpha^{\pm}$.  We have
 \beq
\Irr(\mb D_n)= \{ \Res^{\mathbb{B}_n}_{\mathbb{D}_n} \rho_{\alpha,\beta} \mid (\alpha,\beta) \models n, \alpha \neq \beta \} \coprod \left\{ \rho_{\alpha}^{+}, \rho_{\alpha}^- \mid \alpha \vdash \frac{n}{2} \right\},
 \eeq
 with the second set nonempty only when $n$ is even.  The total number of representations of $\mathbb{D}_n$ is therefore equal to $\half p_2(n)$ when $n$ is odd, and equal to
 \beq
 \half p_2(n)+ \frac{3}{2} p\left(\frac{n}{2} \right)
 \eeq
 when $n$ is even.  In particular, there are two multiplicative characters of $\mathbb{D}_n$: the trivial character, and the restriction of $\sgn^0$ from $\mb B_n$, which we denote by `$\sgn$'.  
 
   The determinants of the Type I representations are given by 
 \beq
 \det \Res^{\mathbb{B}_n}_{\mathbb{D}_n} \rho_{\alpha \beta}=(\sgn)^{y_{\alpha \beta}}.
 \eeq
 
 So if $n$ is odd, we have
 \beq
 N_1'(n)=\half(N_1(n)+N_\eps(n))
 \eeq
 and
\beq
 N_{\sgn}'(n)=\half( N_{\sgn^0}(n)+N_{\sgn^1}(n)).
 \eeq
 (The $N_\omega(n)$ were computed in the previous section.)

 \subsection{Type II Determinants}
  Let $n=2a$.  For $\alpha \vdash a$, we have
  $\det  \Res^{\mathbb{B}_n}_{\mathbb{D}_n} \rho_{\alpha,\alpha} =1$ by  Proposition \ref{hookss}.  It follows that $\det \rho_\alpha^+=\det \rho_\alpha^-$, so we need only compute $\det \rho_\alpha^+$.
  
Put
 \beq
 x_\alpha^{+}=\frac{\dim \rho_\alpha^{+}-\chi_\alpha^+(s_1)}{2} \in \Z/2\Z.
 \eeq
  By Proposition \ref{solomon.principle}, 
 \beq
 \det \rho_\alpha^{+}=(\sgn)^{x_\alpha^+}.
 \eeq
  
If $n \geq 3$, it is not hard to see that $\Int(e_1)(s_1)=e_1 s_1 e_1^{-1}$ is conjugate in $\mb D_n$ to $s_1$; it follows that $\chi_\alpha^+(s_1)=\chi_\alpha^-(s_1)$.
 
Using Equation (\ref{charles}) we compute
 \beq
 \begin{split}
 \chi_\alpha^+(s_1) &=\half \chi_{\rho_{\alpha\alpha}}(s_1) \\
 			 &= \binom{n-2}{a,a-2} f_\alpha \chi_\alpha(s_1) \\
			 &= \binom{n-2}{a,a-2} f_\alpha (f_\alpha-2 g_\alpha), \\
				\end{split}
				\eeq
  and so
  \beq
 x_{\alpha}^+=\frac{1}{4}  \binom{n}{a,a} f_\alpha^2-\half  \binom{n-2}{a-2,a} f_\alpha^2+\binom{n-2}{a,a-2} f_\alpha g_\alpha.
 \eeq
 
  If $f_\alpha$ is even, then clearly $x_\alpha^+$ is even.   If $f_\alpha$ is odd, then
\beq{}
\begin{split}
 x_{\alpha}^+ &=\frac{1}{4}  \binom{n}{a,a}-\half  \binom{n-2}{a-2,a}+\binom{n-2}{a,a-2}g_\alpha \\
 		&= \half \binom{n-2}{a-1,a-1}+\binom{n-2}{a,a-2}g_\alpha. \\
		\end{split}
 \eeq{}

 \begin{lemma}
Let $n=2a \geq 6$ be an even integer.
\begin{enumerate}
\item If $n$ is neither of the form $2^k$ nor $2^k+2$, then  $\binom{n-2}{a-1,a-1}$ is a multiple of $4$, and $\binom{n-2}{a,a-2}$ is even.
\item If $n$ is of the form $2^k$, then $\binom{n-2}{a-1,a-1}$ is a multiple of $4$, and $\binom{n-2}{a,a-2}$ is odd.
\item If $n$ is of the form $2^k+2$, then $\half \binom{n-2}{a-1,a-1}$ is odd and $\binom{n-2}{a,a-2}$ is even.
\end{enumerate}
\end{lemma}
 
The proof is left to the reader. $\square$ 
 
\begin{prop} Let $n\geq 6$ be an even integer.
\begin{enumerate}
\item If $n$ is not of the form $2^k$ or $2^k+2$, then all Type II representations have determinant $1$.
\item If $n$ is of the form $2^k$, then $\rho_{\alpha}^+, \rho_{\alpha}^-$ have determinant $\sgn$ iff $f_\alpha$ and $g_\alpha$ are odd.
\item If $n$ is of the form $2^k+2$, then $\rho_{\alpha}^+, \rho_{\alpha}^-$ have determinant $\sgn$ iff $f_\alpha$ is odd.
\end{enumerate}
 \end{prop}

\begin{thm} Let $n \geq 4$.
 We have
 \beq
 N_{\sgn}'(n)=
  \left\{
\begin{array}{ll}
 \half( N_{\sgn^0}(n)+N_{\sgn^1}(n)) +\half  n&\text{ if $n=2^k$ for some $k\geq1$,}\\
 \half( N_{\sgn^0}(n)+N_{\sgn^1}(n))+n-2 & \text{ if $n=2^k+2$ for some $k\geq 1$},\\
 \half( N_{\sgn^0}(n)+N_{\sgn^1}(n)) & \text{ otherwise.}
 \end{array}
  \right.  
 \eeq
 \end{thm}
 
 This follows from the above and using $A(n/2)=\half n$ when $n$ is a power of $2$, and $A(n/2)=\half n -1$ when $n=2^k+2$.  (The case $n=4$ is easily done separately, and agrees with the first two cases.)
 
 Note that if $n$ is even, then
 \beq
 N_1'(n)+  N_{\sgn}'(n)=\half p_2(n)+\frac{3}{2}p(n),
 \eeq
 so in principle we have a formula for $N_1'(n)$ as well.
 
 \section{Exceptional Coxeter Groups}
 
 There are only finitely many irreducible finite Coxeter groups not of type $A_n$, $B_n$, $D_n$, or $I_2(n)$, namely the exceptional ones.   For completeness, we give $N_\omega$ for each exceptional group.  This is easily done by the propositions in Section \ref{SolomonP}, together with the character tables for these groups, which may be found in \cite{Geck}.
 
 For the exceptional Coxeter groups $W$ with $|W_{\ab}|=2$, we offer the following table.  Below $N_1$ is the number of irreducible representations of $W$ with trivial determinant, and $N_{\eps_W}$ is the number with $\eps_W$ as determinant.

\begin{center}
\begin{tabular}{|c|c|c|} 
 \hline
Type & $N_{1}$ & $N_{\eps_W}$  \\ 
 \hline\hline
 $H_3$ & $6$ & $4$ \\
 \hline
 $H_4$ & $19$ & $15$ \\
 \hline 
 $E_6$ & $13$ & $12$ \\
 \hline
 $E_7$ & $44$ & $16$ \\
 \hline
 $E_8$ & $63$ & $49$ \\
 \hline
 \end{tabular}
\end{center}

\bigskip

Finally, for type $F_4$, we have $N_1=9$, $N_{\eps_W}=8$, and $N_\omega=4$ for $\omega \neq 1,\eps_W$.

 \bibliographystyle{alpha}
\bibliography{refs}

\begin{thebibliography}{TSCc08}

\bibitem[APS17]{APS-chiral}
Arvind Ayyer, Amritanshu Prasad, and Steven Spallone.
\newblock Representations of symmetric groups with non-trivial determinant.
\newblock {\em J. Combin. Theory Ser. A}, 150:208--232, 2017.

\bibitem[BH06]{Bushnell.Henniart}
Colin~J. Bushnell and Guy Henniart.
\newblock {\em The local {L}anglands conjecture for {$\rm GL(2)$}}, volume 335
  of {\em Grundlehren der Mathematischen Wissenschaften [Fundamental Principles
  of Mathematical Sciences]}.
\newblock Springer-Verlag, Berlin, 2006.

\bibitem[Bou02]{Bou.Lie.4-6}
Nicolas Bourbaki.
\newblock {\em Lie groups and {L}ie algebras. {C}hapters 4--6}.
\newblock Elements of Mathematics (Berlin). Springer-Verlag, Berlin, 2002.
\newblock Translated from the 1968 French original by Andrew Pressley.

\bibitem[GP00]{Geck}
Meinolf Geck and G\"otz Pfeiffer.
\newblock {\em Characters of finite {C}oxeter groups and {I}wahori-{H}ecke
  algebras}, volume~21 of {\em London Mathematical Society Monographs. New
  Series}.
\newblock The Clarendon Press, Oxford University Press, New York, 2000.

\bibitem[HR00]{Hardy.Ram}
G.~H. Hardy and S.~Ramanujan.
\newblock Asymptotic formul\ae \ in combinatory analysis [{P}roc. {L}ondon
  {M}ath. {S}oc. (2) {\bf 17} (1918), 75--115].
\newblock In {\em Collected papers of {S}rinivasa {R}amanujan}, pages 276--309.
  AMS Chelsea Publ., Providence, RI, 2000.

\bibitem[Mac71]{macdonald}
Ian~G. Macdonald.
\newblock On the degrees of the irreducible representations of symmetric
  groups.
\newblock {\em Bulletin of the London Mathematical Society}, 3(2):189--192,
  1971.

\bibitem[Ols93]{olsson}
J{\o}rn~B. Olsson.
\newblock {\em Combinatorics and representations of finite groups}, volume~20
  of {\em Vorlesungen aus dem Fachbereich Mathematik der Universit\"at GH
  Essen}.
\newblock Universit\"at Essen, 1993.

\bibitem[Pak06]{Igor}
Igor Pak.
\newblock Partition bijections, a survey.
\newblock {\em Ramanujan J.}, 12(1):5--75, 2006.

\bibitem[Rob55]{Robbins}
Herbert Robbins.
\newblock A remark on {S}tirling's formula.
\newblock {\em Amer. Math. Monthly}, 62:26--29, 1955.

\bibitem[Sta01]{ec2}
Richard~P. Stanley.
\newblock {\em Enumerative Combinatorics}, volume~2.
\newblock Cambridge University Press, 2001.

\bibitem[TSCc08]{Sage-Combinat}
{T}he {S}age-{C}ombinat community.
\newblock {S}age-{C}ombinat: enhancing {S}age as a toolbox for computer
  exploration in algebraic combinatorics, 2008.
\newblock \url{http://combinat.sagemath.org}.

\bibitem[TSD15]{sage}
{T}he~{S}age {D}evelopers.
\newblock {\em {S}age {M}athematics {S}oftware ({V}ersion 6.10)}, 2015.
\newblock \url{http://www.sagemath.org}.

\end{thebibliography}
 \end{document}